\documentclass[11pt]{amsart}
\usepackage{latexsym, amsmath, amssymb, amsthm, mathrsfs}
\usepackage[alwaysadjust]{paralist}
\usepackage{caption}
\usepackage{subcaption}
\usepackage[T1]{fontenc}
\usepackage{lmodern}
\usepackage[backref,colorlinks=true,linkcolor=blue,urlcolor=blue, citecolor=blue]%
  {hyperref}
\usepackage[shortalphabetic]{amsrefs}
\usepackage[all]{xy}
\usepackage[T1]{fontenc}
\usepackage{mathptmx}
\usepackage{microtype}
\usepackage{framed}
\usepackage{tikz}
\usepackage{graphicx}
\usepackage{wasysym}
\usepackage[alwaysadjust]{paralist}

\makeatletter
\providecommand\@enum@widestlabel{7}
\makeatother

\usepackage[centering, includeheadfoot, hmargin=0.8in, vmargin=1in,
  headheight=30.4pt]{geometry}
\newtheorem{lemma}{Lemma}[section]
\newtheorem{theorem}[lemma]{Theorem}
\newtheorem{corollary}[lemma]{Corollary}
\newtheorem{proposition}[lemma]{Proposition}
\newtheorem{conjecture}[lemma]{Conjecture}
\theoremstyle{plain}

\theoremstyle{definition}

\newtheorem{remark}[lemma]{Remark}
\newtheorem{example}[lemma]{Example}

\renewcommand{\theequation}%
{\arabic{section}.\arabic{lemma}.\arabic{equation}}

\newcommand{\CC}{\ensuremath{\mathbb{C}}} 
\newcommand{\NN}{\ensuremath{\mathbb{N}}} 
\newcommand{\PP}{\ensuremath{\mathbb{P}}} 
\newcommand{\QQ}{\ensuremath{\mathbb{Q}}} 
\newcommand{\RR}{\ensuremath{\mathbb{R}}} 
 
\newcommand{\sI}{\ensuremath{\kern -1pt \mathscr{I}\kern -2pt}} 
\newcommand{\sJ}{\ensuremath{\kern -2pt \mathscr{J}\kern -2pt}} 
\newcommand{\sO}{\ensuremath{\mathscr{O}}} 
\newcommand{\sK}{\ensuremath{\mathscr{K}}}

\renewcommand{\geq}{\geqslant}
\renewcommand{\leq}{\leqslant}

\DeclareMathOperator{\mult}{mult}

\DeclareMathOperator{\Sym}{Sym}

\DeclareMathOperator{\ord}{ord}
\DeclareMathOperator{\Null}{Null}
\DeclareMathOperator{\Neg}{Neg}
\DeclareMathOperator{\vol}{vol}
\DeclareMathOperator{\length}{length}

\DeclareMathOperator{\pr}{pr}
\DeclareMathOperator{\Area}{Area}
\DeclareMathOperator{\dist}{dist}
\DeclareMathOperator{\intt}{interior}

\newcommand{\equ}{\ensuremath{\,=\,}}
\newcommand{\dgeq}{\ensuremath{\,\geq\,}}
\newcommand{\dleq}{\ensuremath{\, \leq\, }}
\newcommand{\deq}{\ensuremath{\stackrel{\textrm{def}}{=}}}

\newcommand{\dsubseteq}{\ensuremath{\,\subseteq\,}}
\newcommand{\st}[1]{\ensuremath{\left\{ #1 \right\}   }}
\newcommand{\zj}[1]{\ensuremath{\left( #1 \right)}}
\newcommand{\dsupseteq}{\ensuremath{\,\supseteq\,}}
\newcommand{\lra}{\ensuremath{\longrightarrow}}
\newcommand{\og}{\ensuremath{\overline{\Gamma}}}         
\newcommand{\e}{\ensuremath{\epsilon}}

\newcommand{\HH}[3]{\ensuremath{H^{#1}\left(#2,#3\right)}}

\newcommand{\iss}[1]{\ensuremath{\Delta^{-1}_{#1}}}
\newcommand{\eone}{\ensuremath{\textup{\textbf{e}}_1}}

\begin{document}

\title[Reider-type theorem for higher syzygies]{A Reider-type theorem for higher syzygies on abelian surfaces}

\author[A.~K\" uronya]{Alex K\" uronya}
\author[V.~Lozovanu]{Victor Lozovanu}

\address{Alex K\"uronya, Johann-Wolfgang-Goethe Universit\"at Frankfurt, Institut f\"ur Mathematik, Robert-Mayer-Stra\ss e 6-10., D-60325
Frankfurt am Main, Germany}
\address{Budapest University of Technology and Economics, Department of Algebra, Egry J\'ozsef u. 1., H-1111 Budapest, Hungary}
\email{{\tt kuronya@math.uni-frankfurt.de}}

\address{Victor Lozovanu, Universit\'a degli Studi di Milano--Bicocca, Dipartimento di Matematica e Applicazioni, 
Via Cozzi 53,, I-20125 Milano, Italy}
\email{{\tt victor.lozovanu@unimib.it}}
\email{{\tt victor.lozovanu@gmail.com}}

\begin{abstract}
Building on the theory of infinitesimal Newton--Okounkov bodies and previous work of  Lazarsfeld--Pareschi--Popa, we present a Reider-type theorem for higher syzygies of ample line bundles on abelian surfaces. 
As an application of our methods we  confirm a conjecture of Gross and Popescu on abelian surfaces with a very ample primitive polarization of type $(1,d)$, whenever $d\geq 23$. 
\end{abstract}

\maketitle

\section{Introduction}

\subsection{Background and motivation}
A very effective way to study  varieties is through their embeddings into projective space. For a projective variety $X$, a very ample line bundle $L$ over $X$ gives rise 
to an embedding
\[
\phi_L \colon X \ \hookrightarrow \ \PP^N \equ \PP(H^0(X,L)) \ . 
\]
The homogeneous coordinate ring of the image, defined as $R(X,L)=\oplus_{m\in\NN} H^0(X,L^{\otimes m})$, is  the main algebraic invariant associated to the pair $(X,L)$. 
Quite naturally, the algebraic properties of $R(X,L)$  reveal a lot of information about the geometric picture it is attached to, and it has long been an object of 
central significance. 

The algebraic behaviour of $R(X,L)$ is best studied in the category of graded modules over the ring $S=\CC[x_0,\ldots ,x_N]=\Sym(H^0(X,L))$.  As an  $S$-module, 
$R(X,L)$ admits a minimal graded free resolution $E_{\bullet}$
\[
\ldots \rightarrow E_i \rightarrow E_{i-1}\rightarrow \ldots \rightarrow E_2\rightarrow E_1\rightarrow E_0\rightarrow R_l \rightarrow 0 \ ,
\]
where each $E_i=\oplus_{j}S(-a_{i,j})$ is a free graded $S$-module. Beside its  geometric relevance, studying the set of numbers $(a_{i,j})$  is of paramount interest 
in algebra, combinatorics, and  related areas.

Following the footsteps  of Castelnuovo  and later  Mumford \cite{Mumford}, Green and Lazarsfeld \cite{GL}  introduced a sequence of properties asking  
that the first $p$ terms in $E_{\bullet}$  be as simple as possible. More specifically, one says that the line bundle $L$ \emph{satisfies property} $(N_p)$ if $E_0=S$ and 
\[
a_{i,j} \equ i+1 \ \ \ \ \text{for all $j$ and all $1\leq i\leq p$.}
\]
As an illustration, translated to geometric terms, property $(N_0)$ holds for $L$ if and only if $\phi_L$ defines a projectively normal embedding, while  property $(N_1)$ 
is equivalent to requiring $(N_0)$  and that the homogeneous ideal $\sI_{X|\PP^N}$ of $X$ in $\PP^N$ be  generated by quadrics. 
Historically, property $(N_0)$ on curves was first studied by Castelnuovo. Many years later  Mumford completed the picture in the curve case for  $(N_0)$ and $(N_1)$. 
 
Due to its classical roots and its relevance for projective geometry, the area surrounding  property  $(N_p)$ has generated a significant amount of  
work in the last thirty years or so, with some of the highlights being \cites{BEL, EL1, EL2,P,Voisin}. Controlling higher syzygies 
has always been  a notoriously difficult question. For more details about this circle of ideas, the reader is kindly referred  to \cite{PAGI}*{Section 1.8.D}  and \cite{E}. 

In analogy to Fujita's influential conjectures on global generation and very ampleness, Mukai formulated a series of conjectures for properties $(N_p)$. For an ample divisor $L$ on a smooth projective variety $X$, Mukai asks if property $(N_p)$ holds for $K_X+mL$ whenever $m\geq \dim X+p+2$.
From this perspective the case of curves had been well understood by Green \cite{G}. Although non-trivial work has been done in dimension two (see \cite{GP}
for a nice overview), the conjecture in its entirety currently appears  to be  out of reach even for surfaces. 

The only class of varieties where Fujita-Mukai conjecture has been successfully treated  is that of abelian varieties. Based on earlier work of Kempf, 
Pareschi \cite{P} proved a conjecture of Lazarsfeld claiming that for an ample line bundle $L$ on an abelian variety $X$, the line bundle  $L^{\otimes (p+3)}$ satisfies 
property $(N_p)$. Note  that the Fujita--Mukai conjectures 'do not scale well' as seen by replacing $L$ by a large multiple. Koll\'ar's suggestion 
\cite{SingsPairs} to  control positivity in terms of intersection numbers seems closer to the truth. 

Motivated by  Koll\'ar's line of thought, it is  natural to ask whether it is feasible to  study property $(N_p)$ for a given line bundle with certain numerics instead. 
Such  a new line of attack in the case of abelian varieties has been recently initiated by Hwang--To \cite{HT} where complex analytic techniques 
(more precisely upper bounds on volumes of tubular neighbourhoods of subtori of abelian varieties) were  used  to control projective normality of line bundles on 
abelian varieties in terms of Seshadri constants. 

Shortly after, Lazarsfeld--Pareschi--Popa \cite{LPP}  used multiplier ideal methods to extend the results of \cite{HT}  to  higher syzygies on abelian varieties. 
The essential contribution of \cite{LPP} is that property $(N_p)$ can be guaranteed via constructing effective divisors with  prescribed multiplier 
ideals and numerics. Our approach builds in part  on the method of proof developed in  \cite{LPP}.

\subsection{Description of the  main result}
The goal of this article is to study property $(N_p)$ for line bundles on abelian surfaces over the complex numbers. Based on ideas coming  from algebraic and differential 
geometry aided by  the convex geometry of polygons in the plane, we are able to prove the following theorem. 

\begin{theorem}\label{thm:np}
Let $p\geq 0$ be a natural number, $X$  a complex abelian surface,  and $L$ an ample line bundle on $X$ with $(L^2)\geq 5(p+2)^2$. 
Then the following are equivalent. 
\begin{enumerate}
\item $X$ does not contain an elliptic curve $C$ with $(C^2)=0$ and $1\leq (L\cdot C) \leq p+2$.
\item The line bundle $L$ satisfies property $(N_p)$.
\end{enumerate} 
\end{theorem}

The sequence of properties $(N_p)$ is best considered as increasingly stronger algebraic versions of positivity for line bundles along the lines of global generation
and very ampleness. From this point of view, Theorem~\ref{thm:np} is a natural generalization of Reider's celebrated result \cite{R}. 

More concretely, in the case of an ample line bundle on an abelian surface Reider's theorem says: if $(L^2)\geq 5$ then $L$ is globally generated 
if and only if there is no elliptic curve $C\subseteq X$ with $(C^2)=0$ and $(L\cdot C)=1$. Similarly,  if $(L^2)\geq 10$, then $L$ is  very ample 
exactly if $X$ does not contain an elliptic curve $C$ with $(C^2)=0$ and $1\leq (L\cdot C)\leq 2$. It is worth mentioning here that for a general abelian surface $X$, 
a principal polarization $L$ of type $(1,6)$ for instance  is very ample by Reider's theorem but  the embedding it defines is not projectively normal 
(see Remark~\ref{rem:veryample vs normality}). Hence the small discrepancy between Reider's theorem and Theorem~\ref{thm:np} in the case of $p=0$, 
i.e. projectively normality.

In comparison with  the proofs of Reider's and Pareschi's theorem,  both relying heavily on vector bundle techniques, our approach in confirming property
$(N_p)$ relies on  multiplier ideals and  the associated vanishing theorems together with the theory of infinitesimal Newton--Okounkov polygons  developed in \cite{KL}. 
The essential novelty of our work is the use of  infinitesimal Newton--Okounkov polygons to construct effective $\QQ$-divisors  
whose multiplier ideal coincides with  the maximal ideal of the origin.

The strategy  of using  effective divisors with prescribed singularities at a given point dates back to the work of Demailly in the early 1990's,  and it has proven 
to be crucial in many groundbreaking results. To name a notable example,  the work of Anghern and Siu on the Fujita conjecture relies heavily on this idea. 

On the other hand, when the given point is taken to be very general,  it is known that the  line bundle in question will have very strong local positivity properties. 
For example, Ein, K\"uchle, and Lazarsfeld (see \cite{EL} and \cite{EKL}) proved that the Seshadri constant of an ample line bundle at a very general point is quite large. 
Based on earlier  work of Nakamaye, we have translated these ideas to the language of infinitesimal Newton--Okounkov bodies in the case of surfaces (cf. \cite{KL}). 
Recall that as far as local properties go, the origin of an abelian surface behaves like a very generic point. 

The above principles lead to the following criterion for finding effective $\QQ$-divisors with prescribed singularities 
by making use of the Euclidean geometry of infinitesimal Newton--Okounkov polygons; this is the main ingredient of  the proof of the implication $(1)\Rightarrow (2)$ 
of Theorem~\ref{thm:np}.

\begin{theorem}(Theorem~\ref{thm:verygeneric} and Corollary~\ref{cor:very general})
Let $p\geq 0$ be a positive integer, $X$ a smooth projective surface and $L$ an ample line bundle on $X$ with $(L^2)\geq 5(p+2)^2$. Let $x\in X$ be a very general point so that there is no irreducible curve $C\subseteq X$ smooth at $x$ with $1\leq (L\cdot C)\leq p+2$. Then there exists an effective $\QQ$-divisor 
\[
D \ \equiv \ \frac{(1-c)}{p+2}L, \textup{ for some } 0<c\ll 1
\]
such  that $\sJ(X;D)=\sI_{X,x}$ in a neighborhood of the point $x$, where $\sJ(X;D)$ denotes the multiplier ideal of $D$.
\end{theorem}

The implication $(2)\Rightarrow (1)$ of Theorem~\ref{thm:np} on the other hand is achieved  by the  method of restricting syzygies introduced in  \cite{GLP} 
and developed further by the authors of  \cite{EGHP}. More precisely, the  strategy  is that  property $(N_p)$ for the line bundle $L$ implies vanishing of 
certain higher cohomology groups of  the ideal sheaf of the scheme-theoretical  intersection $X\cap \Lambda$, where $\Lambda$ is a plane of small dimension 
inside the projective space  $\PP(\HH{0}{X}{L})$. 

The main result of \cite{LPP} in dimension two follows from our Theorem~\ref{thm:np}. Part of the added value of our work lies in treating  the cases where $(L^2)$ is large, but 
$\epsilon(L;o)$ is small; as seen in \cite{BS} (cf.  Example~\ref{eg:Koszul}), such line bundles abound. Along the same lines, we give a criterion for the section ring 
$R(X,L)$ of an ample line bundle $L$ on an abelian surface to be  Koszul (see Corollary~\ref{cor:Koszul}).

\subsection{Consequences}
Before giving a sketch of the proof of Theorem~\ref{thm:np},  we  record here a couple of interesting applications. Perhaps the most significant one is the following positive answer to a conjecture of Gross and Popescu \cite{GrP} regarding the generating degrees  of the ideal sheaf of an embedded abelian surface with a very ample primitive polarization of type $(1,d)$, proven in detail in Section 5. 

\begin{theorem}(Theorem~\ref{thm:grosspopescu})
Let $X$ be an abelian surface, $L$ a very ample line bundle  of type $(1,d)$ on $X$ for some $d\geq 23$. Then $L$ defines an embedding 
$X\subseteq \PP^N=\PP(H^0(X, L))$ whose ideal sheaf $\sI_{X|\PP^N}$  is globally generated by quadrics and cubics. 
\end{theorem}

The next application has a more classical flavour as it deals with  multiples of an ample divisor $L$.

\begin{corollary}\label{cor:one}
Let $X$ be an abelian surface,  $L$ an ample line bundle on $X$ with $(L^2)\geq 5$. Then
\begin{enumerate}
\item The line bundle $L^{\otimes (p+3)}$ satisfies property $(N_p)$.
\item The line bundle $L^{\otimes (p+2)}$ satisfies property $(N_p)$ if and only if  $(X, L)\ncong (C_1\times C_2, L_1\boxtimes L_2)$, 
where $L_1$ is a principal polarization of the elliptic curve $C_1$ and $L_2$ is of type $(d)$ on the elliptic curve $C_2$.
\end{enumerate}
\end{corollary}

To see this, by  \cite{Nak}*{Lemma 2.6} the existence of an elliptic curve $C_1\subseteq X$ with $(L\cdot C_1)=1$ is equivalent to  $X$ being  
the product of $C_1$ with another elliptic curve $C_2$, and $L\simeq\sO_{C_1\times C_2}(P, D)$, where $P\in C_1$ a point and $D$ a divisor on $C_2$. 
Taking this into account,  Corollary~\ref{cor:one} is an immediate  consequence of Theorem~\ref{thm:np}.

Corollary~\ref{cor:one} recovers, at least in dimension two, Pareschi's result mentioned above under the mild additional condition that $(L^2)\geq 5$. 
Based on their theory of $M$-regularity on abelian varieties, Pareschi and Popa in \cite{PP} obtain  that for an ample line bundle $L$ with 
no fixed components on an abelian variety $X$,  $L^{\otimes(p+2)}$ satisfies property $(N_p)$. Hence  Corollary~\ref{cor:one} $(2)$ is a numerical counterpart of 
\cite{PP}*{Theorem 6.2} in the surface case. 

We point out that Corollary~\ref{cor:one} and \cite{PP}*{Theorem 6.2}  are in fact quite close to each other in dimension two. 
Recall that any ample line bundle $L$ on an abelian surface $X$ is of polarization $(d_1,d_2)$, for some $d_1,d_2\in\NN$ with $d_1$ dividing $d_2$. 
If $d_1\geq 2$, then $L$ is base point free with $(L^2)\geq 8$. Thus, Corollary~\ref{cor:one} $(2)$ delivers  the same result as \cite{PP}*{Theorem 6.2}. 
If $d_1=1$, then by the Decomposition theorem \cite{BL}*{Theorem~4.3.1} the linear series  $|L|$ has base curves if and only if  
$(X, L)\equiv (C_1\times C_2, L_1\boxtimes L_2)$, where $L_1$ is a principal polarization on $C_1$ and $L_2$ is of type $(d)$ on $C_2$. 
Thus,  line bundles of polarization $(1,2)$ are the only case not covered by Corollary~\ref{cor:one} $(2)$, but still within the scope of the work of Pareschi and Popa.

The third consequence of Theorem~\ref{thm:np} deals with $k$-very ampleness, another notion of strong positivity. One of the applications of 
\cite{EGHP} is that property  $(N_p)$ implies  $(p+1)$-very ampleness (a line bundle $L$ is $k$-very ample if the restriction map 
$H^0(X,L)\rightarrow H^0(L|_Z)$ is  surjective for any $0$-dimensional subscheme $Z\subseteq X$ of length at most $k+1$). 
By \cite{EGHP}*{Remark 3.9} one obtains the following. 

\begin{corollary}\label{cor:two}
Under the assumptions of Theorem~\ref{thm:np}, the following two conditions are equivalent:
\begin{enumerate}
\item $X$ does not contain an elliptic curve $C$ with $(C^2)=0$ and $1\leq (L\cdot C) \leq p+2$.
\item The line bundle $L$ is $(p+1)$-very ample.
\end{enumerate} 
\end{corollary}

The implication $(2)\Rightarrow (1)$ can also be derived from Terakawa's paper \cite{Ter}; the  condition  $(L^2)\geq 5(p+2)^2$ yields that  
the  result of Terakawa is in fact equivalent to Corollary~\ref{cor:two}. On the other hand, the argument in \cite{Ter} relies on Reider's theorem in an 
essential way, hence again makes crucial use of vector bundle techniques. 
 
The main results on $k$-very ampleness on abelian surfaces  are due to Bauer and Szemberg. In \cite{BS1} they treat  powers of an ample line bundle on 
an arbitrary abelian variety, while in  \cite{BS2} they tackle the case of a primitive line bundle on an abelian surface with Picard number one.
Again, the novelty of Corollary~\ref{cor:two} is that neither does it require $L$ to be a high multiple of an ample line bundle, nor does it put a restriction 
on the underlying geometry of $X$.  Note also that  $k$-very ampleness has been successfully studied  using derived category methods 
by  Arcara--Bertram (cf. \cite{AB}*{Corollary~3.9.(b)}),  Alagal--Maciocia \cite{AM}, and again by  Pareschi--Popa  \cite{PP}.

\subsection{Sketch of the proof}
We outline the ideas behind the implication $(1)\Rightarrow (2)$ in Theorem~\ref{thm:np} for the case of  projective normality, i.e. $p=0$. To this end, let 
$X$ be an abelian surface over the complex numbers and $L$  an ample line bundle on $X$ with $(L^2)\geq 20$.  Suppose in addition that $X$ does not contain 
any elliptic curve $C$ with $(C^2)=0$ and $1\leq (L\cdot C)\leq 2$. To ease the presentation we assume that there exists a 
Seshadri-exceptional  curve $F\subseteq X$  through the origin $o\in X$ such  that $r\deq (L\cdot F)\geq q=\mult_o(F)\geq 2$ and 
$\epsilon\deq \epsilon(L;o)=r/q$. It is worth pointing out  here that it is the case $q=1$ when elliptic curves of small $L$-degree occur.

Our starting point is the method of  \cite{LPP}, which builds on the following observation: consider the diagonal $\Delta\subseteq X\times X$ 
with ideal sheaf $\sI_{\Delta}$; as shown in \cite{I},  projective normality of $L$ is equivalent to the vanishing condition
\begin{equation}\label{eq:intro Green}
H^1(X\times X, \textup{pr}_1^*(L)\otimes \textup{pr}_2^*(L^{\otimes h})\otimes\sI_{\Delta}) \equ 0\  \text{for all $h\geq 1$}.
\end{equation}
The authors of \cite{LPP} then go on to show that in order to guarantee the vanishing in (\ref{eq:intro Green}), it suffices to verify 
the existence of an effective  $\QQ$-divisor 
\[
D\ \equiv \ \frac{1-c}{2}L, \textup{ for some } 0<c\ll 1 
\]
satisfying the  additional property that $\sJ(X,D)=\sI_{X,o}$. Assuming one can do so, using  the difference morphism $\delta\colon X\times X\rightarrow X$ 
given  by  $\delta (x,y)=x-y$, one deduces that  $\sI_{\Delta}=f^*(\sJ(X;D))=\sJ(X\times X,f^*(D))$, which in turn leads to (\ref{eq:intro Green}) via 
Nadel vanishing for multiplier ideals. 

While directly constructing divisors with a given multiplier ideal is quite difficult in general, a simple observation 
from homological algebra (see Corollary~\ref{cor:normality}) ensures that at least in the case of projective normality  
 it is enough to exhibit  such a divisor $D$ with $\sJ(X,D)=\sI_{X,o}$ locally around the origin $o\in X$. 
 
This is where the main new ingredient of the paper comes into play: it turns out that one can use  infinitesimal Newton--Okounkov polygons to show 
the existence of   suitable $\QQ$-divisors $D$ with $\sJ(X,D)=\sI_{X,o}$ over an open subset containing $o$. 

Write  $\pi:X'\rightarrow X$ for the blowing-up of $X$ at the origin $o$ with exceptional divisor $E$, and let $B\deq\frac{1}{2}L$. The first step is 
to find a criterion in terms of infinitesimal Newton--Okounkov polygons that guarantees the existence of divisors  $D$ as above. In Theorem~\ref{thm:nopolygon}
 we show that if
\[
\intt \Big(\Delta_{(E,z)}(\pi^*(B)) \ \cap \ \underbrace{\{(t,y)\ | \ t\geq 2, 0\leq y\leq \frac{1}{2}t\}}_{\Lambda}\Big)
\] 
is non-empty for any $z\in E$, then one will always find an effective $\QQ$-divisor $D=(1-c)B$ with $\sJ(X;D)=\sI_{X,o}$ in a neighborhood of  $o$.

Aiming at a contradiction suppose  that for some $z_0\in E$ the Newton--Okounkov polygon $\Delta_{(E,z_0)}(\pi^*(B))$ does not intersect the interior of 
the region $\Lambda$ (for an illustration see Figure~\ref{fig:1} (a)). This implies that the polygon $\Delta_{(E,z_0)}(\pi^*(D))$ sits above a certain line 
that passes through the point $(2,1)$. But since the area of  $\Delta_{(E,z_0)}(\pi^*(B))$ is quite big (equal to $(B^2)/2\geq 5/2$),  the Seshadri constant
 $\epsilon(B;o)$ is then forced to be  small by convexity, for it is  equal to the size of the largest inverted simplex inside $\Delta_{(E,z_0)}(\pi^*(B))$ 
by \cite{KL}*{Theorem 3.11}.  A more precise computation gives  the upper bound  $\epsilon(B;o)\leq \frac{5-\sqrt{5}}{2}$. 

In order to obtain  a contradiction, notice that $X$ carries a transitive group action, thus the origin $o\in X$ behaves like a very general point. 
Relying on  \cite{EKL} and \cite{NakVeryGen},  the authors show in \cite{KL}*{Proposition 4.2}   the  inclusion  
\[
\Delta_{(E,z)}(\pi^*(B)) \ \subseteq \ \triangle OAC, \text{ for generic point }z\in E \ ,
\]
where $O=(0,0),A=(r/2q,r/2q)$ and $C=r/2(q-1)$ (see Figure~\ref{fig:1} (b)). The area of the polygon on the left-hand side  is  $(B^2)/2\geq 5/2$, hence a simple area comparison gives  
$\epsilon(B;o)=\frac{1}{2}\epsilon(L;o)\geq \sqrt{\frac{5}{2}}$, which immediately leads to a contradiction since  
 $\frac{5-\sqrt{5}}{2} < \sqrt{\frac{5}{2}}$.

\begin{figure} 
\begin{tikzpicture}
%
%
\draw [->]  (0,0) -- (4.5,0);   
\node [below right] at (4.5,0) {$t$};
\draw [->]  (0,0) -- (0,4.5); 
\node [left] at (0,4.5) {$y$}; 
\draw (0,0) -- (4.5,4.5); 
\node [left] at (4,4.5) {$y=t$};
\draw [fill=gray, ultra thick] (0,0) -- (2,2) -- (3.5,3) -- (2,0) -- (0,0);  
\node [left] at (2.5,3) {$\Delta_{(E,z_0)}(\pi^*B)$};
\draw (2,0) -- (2,2); 
\draw [ultra thick] (3,0) -- (3,1.5);  
\draw (0,0) -- (5,2.5); 
\node [above] at (5,2.5) {$y=\frac{1}{2}t$};  
\shadedraw [dotted,shading=axis] (3,0) -- (3,1.5) -- (5,2.5) -- (4.25,0);
\node [above] at (4,1) {$\Lambda$}; 
\draw (4.24,0) -- (3,0) -- (3,1.5) -- (5,2.5);
%
%
\node [below left] at (0,0) {$O$};
\node [below] at (2,0) {$(\epsilon,0)$};
\filldraw [black]  (3,1.5) circle (3pt); 
\node [above] at (3.67,1.85) {$(2,1)$};
\end{tikzpicture}
\begin{tikzpicture}
%
%
\draw [->]  (0,0) -- (4.5,0);   
\node [below right] at (4.5,0) {$t$};
\draw [->]  (0,0) -- (0,4.5); 
\node [left] at (0,4.5) {$y$}; 
\draw (0,0) -- (4.5,4.5); 
\node [left] at (4,4.5) {$y=t$};
\draw [fill=gray, ultra thick] (0,0) -- (3,3) -- (3.5,1.5) -- (3.75,0) -- (0,0);  
\node [left] at (2,3) {$\Delta(\pi^*B)$};
\draw (3,0) -- (3,3); 
\draw (3,3) -- (4,0);
%
%
\node [below left] at (0,0) {$O$};
\node [below] at (3,0) {$(\epsilon,0)$};
\node [above] at (3,3) {$A$};
\node [below] at (4,0) {$C$};
\end{tikzpicture} 
\caption{(a) $\Delta_{(E,z_0)}(\pi^*(B))$ and the region $\Lambda$\ \ \ (b) The containment $\Delta_{(E,z)}(\pi^*(B))\subseteq \triangle OAC$}\label{fig:1}
\end{figure}
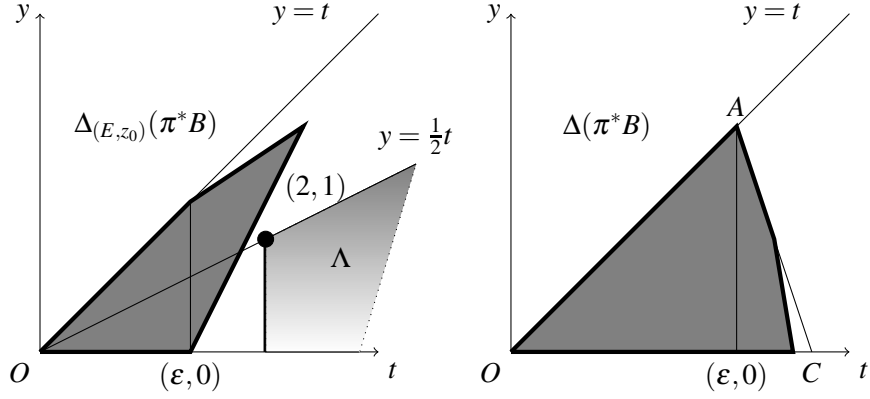

\subsection{Notation and terminology}

In the course of this work, $X$ stands for a smooth projective surface, which from Section 3 onwards will be required to be abelian. The morphism $\pi\colon X'\to X$ 
without exception denotes  the blow-up of a point $x\in X$, which is going to be taken to be the origin whenever $X$ is abelian. A divisor means a $\QQ
$- or $\RR$-Cartier divisor depending on the  context. If we insist that a certain divisor $D$ is integral, we explicitly say so. 

The notation $\Delta_{(C,x)}(D)$ stands for the Newton--Okounkov body $\Delta_{Y_\bullet}(D)$ with $Y_1=C$ and $Y_2=\st{x}$, while $\Delta_x(D)$ denotes 
the generic infinitesimal Newton--Okounkov body $\Delta_{(E,z)}(\pi^*D)$ (see \cite{KL}*{Section 3}).

We work over the complex numbers, although some of our  results might be valid over an arbitrary algebraically closed field. We do not strive for optimal hypotheses. 

\subsection{Organization of the paper}

After a quick recap on infinitesimal Newton--Okounkov bodies, Section 2 deals with constructing singular divisors on arbitrary surfaces at 
arbitrary,  and later,  at very general points. 

Section 3 is devoted to  results specific to abelian surfaces, in particular to the proof of Theorem~\ref{thm:np} along with a description of the method of \cite{LPP},
and a short discussion of the case of projective normality. It is also here that we present a criterion for Koszulness of section rings in terms of self-intersection 
numbers. 

In Section 4 we treat the converse direction of our main theorem: we prove a result that  the existence of low degree elliptic curves on $X$
leads to property $(N_p)$  not being met. 

Finally, in Section 5 we answer the Gross--Popescu conjecture on $(1,d)$-polarized abelian surfaces for $d\geq 23$.

\subsection{Acknowledgements} The authors are grateful to Lawrence Ein, S\'andor Kov\'acs,   and John C. Ottem for helpful discussions. 
We thank Dave Anderson and Mihnea Popa for their valuable   comments on an earlier version of this manuscript. Special thanks are in order to Rob Lazarsfeld 
for his remarks and suggesting a substantial shortcut in Section 4, to Klaus Hulek for his advice regarding \cite{EGHP}, and   to  Thomas Bauer for pointing out 
the class of examples in \cite{BS}*{Example 2.1}. The illustrations were done using the Ti$k$Z package.

\section{Newton-Okounkov polygons and singular divisors}

In this section we give a  recipe for obtaining effective $\QQ$-divisors with prescribed  numerics and   non-trivial multiplier ideal at a given point. 
After a quick overview of infinitesimal Newton--Okounkov polygons  we will focus on a sufficient condition for the  existence of such divisors. We end this section 
with a similar  statement for very general points. 

Note that the results of this section are valid without exception for an arbitrary smooth projective surface $X$; $x\in X$ will be a fixed point, and $\pi\colon X'\to X$ the appropriate blow-up with exceptional divisor $E$.

\subsection{Infinitesimal Newton--Okounkov polygons}

For the general theory of Newton--Okounkov bodies and basic facts we kindly refer the reader to the writings \cites{Bou1,KK,LM}, as far as the two-dimensional 
theory is concerned, the reader is invited to look at \cites{KL,KLM,Roe}. 

Thanks to \cite{LM}*{Theorem 6.2} Newton--Okounkov bodies on smooth projective surfaces are straightforward to determine assuming  that one has 
full information on the variation of Zariski decomposition \cite{BKS}*{Theorem 1.1} in the N\'eron--Severi space of $X$. A more precise analysis using the results
of \cite{BKS} shows that in fact $\Delta_{(C,x)}(D)$ is a polygon with rational slopes and up to possibly two exceptions, rational vertices as well\footnote{One can in 
fact do better, assuming one has control over the  flag. It is shown in \cite{AKL}*{Proposition 11} that given a line bundle $D$, it is always possible 
to arrange for $\Delta_{(C,x)}(D)$ to be a rational polygon.} (see \cite{KLM}*{Theorem B and Proposition 2.2}).

The theory of Newton--Okounkov polygons has been  treated in \cite{KL}, here we will work with so-called \emph{infinitesimal Newton--Okounkov polygons}, convex 
bodies determined by flags coming from exceptional divisors (see \cite{KL15} for the higher-dimensional theory). With notation as above, an infinitesimal Newton--Okounkov
polygon of a divisor $D$ at a point $x\in X$ is a polygon of the form $\Delta_{(E,z)}(\pi^*D$), where $z\in E$ is an arbitrary point. 
The basic convex-geometric objects that play a decisive role for local positivity in terms of infinitesimal data are the \emph{inverted standard simplices}
\[
 \iss{\xi} \deq \st{(t,y)\in\RR^2\mid 0\leq t\leq \xi\, ,\, 0\leq y\leq t}\ .
\]

As it turns out, one has more control over infinitesimal Newton--Okounkov polygons than in general. 

\begin{proposition}[\cite{KL}*{Proposition 3.1}]\label{prop:propinf}
With notation as above, one has 
\begin{enumerate}
\item $\Delta_{(E,z)}(\pi^*D)\subseteq \Delta_{\mu'}^{-1}$ for any $z\in E$;
\item there exist finitely many points $z_1,\ldots ,z_k\in E$ such that the polygon $\Delta_{(E,z)}(\pi^*D)$ is independent of 
$z\in E\setminus\{z_1,\ldots,z_k\}$, with  base  the whole line segment $[0,\mu']\times\{0\}$, 
\end{enumerate}
where $\mu'\deq \sup\st{t>0\mid \pi^*D-tE\text{ is big }}$. 
\end{proposition}

The second part of the Proposition implies that it makes sense to talk about the \emph{generic infinitesimal Newton--Okounkov polygon} of $D$ at the point $x\in X$, 
which we denote by $\Delta_x(D)$. 

Local positivity can be described with the help of infinitesimal flags in a transparent way (cf. \cite{KL15}*{Theorem 3.1} and \cite{KL15}*{Theorem 4.1}). 

\begin{theorem}[\cite{KL}*{Theorem 3.8}]\label{thm:triangle}
Let  $D$ be  a big $\RR$-divisor on a smooth projective surface $X$. Then 
\begin{enumerate}
\item $x\notin \Neg(D)$ if and only if $(0,0)\in\Delta_{(E,z)}(\pi^*D)$ for any $z\in E$,
\item $x\notin \Null(D)$ if and only if there exists $\xi>0$ such that $\Delta_{\xi}^{-1}\subseteq\Delta_{(E,z)}(\pi^*D)$ for any $z\in E$.
\end{enumerate}
\end{theorem}

By Theorem~\ref{thm:triangle} and \cite{KL}*{Lemma 3.14}, it makes sense to define the \emph{largest inverted simplex constant} as follows
\[
 \xi(D;x) \deq  \sup \st{\xi>0\mid \iss{\xi}\subseteq \Delta_{(E,z)}(\pi^*D) }  \ .
\]
In fact it is proven in \cite{KL} that the right hand side does not depend on the choice of the point $z\in E$. Thus the notation makes sense. One of the main statements of \cite{KL} is that using these definitions one can recover quickly the moving Seshadri constant of the divisor $D$ at the point $x$. 

\begin{theorem}[\cite{KL}*{Theorem 3.11}]\label{thm:Seshadri inverted simplex}
Let $D$ be a big  $\RR$-divisor on $X$. If $x\notin\Neg(D)$, then 
\[
\epsilon(||D||;x)  \equ \xi(D;x)\ .
\]
\end{theorem}

As we will see later, local positivity at very general points is somewhat less difficult to control.  
An important observation of Nakamaye's, based on the ideas from \cite{EKL}, is the source of the following convex-geometric estimate, which has manifold applications 
(see \cite{KL}*{Subsection 4.1} for instance). 

\begin{proposition}[\cite{KL}*{Proposition 4.2}]\label{prop:genericinf}
Let $L$ be an ample Cartier divisor on $X$ and  $x\in X$ be a very general point. Then the following mutually exclusive cases can occur. 
\begin{enumerate}
\item $\mu'(L;x)=\epsilon (L;x)$, then $\Delta_x(L)=\Delta^{-1}_{\epsilon(L;x)}$.
\item $\mu'(L,x)>\epsilon (L,x)$, then  there exists an irreducible curve $F\subseteq X$ with $(L\cdot F)=p$ and $\mult_x(F)=q$ such that $\epsilon(L;x)=p/q$. 
Under these circumstances, 
\begin{enumerate}
\item If $q\geq 2$, then $\Delta_x(L)\subseteq \triangle_{ODR}$, where $O=(0,0), D=(p/q,p/q)$ and $R=(p/(q-1),0)$.
\item If $q=1$, then the polygon $\Delta_x(L)$ is contained in the area  below the line $y=t$, and between the horizontal lines 
$y=0$ and $y=\epsilon(L;x)$.
\end{enumerate}
\end{enumerate}
\end{proposition}

\subsection{Singular divisors at arbitrary points} Our purpose  here  is to find an explicit connection between the Euclidean geometry of  infinitesimal Newton--Okounkov polygons and the existence of singular divisors at a given point. Write
\[
\Lambda \deq \{ (t,y)\in \RR^2\ | \ t\geq 2, y\geq 0, \textup{ and }t\geq 2y\}\ . 
\]
Our main result is
\begin{theorem}\label{thm:nopolygon}
Let $B$ be an ample $\QQ$-divisor on $X$. Then we have the following. 
\begin{enumerate}
 \item If \begin{equation}\label{eq:nopolygon1}
\textup{interior of }\big(\Delta_{(E,z)}(\pi^*B) \ \bigcap \ \Lambda\big) \ \neq \ \varnothing, \forall z\in E\ ,
\end{equation}
then there exists an effective $\QQ$-divisor $D\equiv (1-c)B$ for any $0<c\ll 1$ such that $\sJ(X;D)=\sI_{x}$ in a neighborhood of the point $x$.
\item Let us assume that  $(B^2)\geq 5$. If 
\begin{equation}\label{eq:nopolygon2}
\textup{interior of }\big(\Delta_{(E,z_0)}(\pi^*B)\ \bigcap \ \Lambda\big) \ =\ \varnothing\ ,
\end{equation}
for some point $z_0\in E$, then the Seshadri constant $\epsilon(B;x)\leq \frac{5-\sqrt{5}}{2}$.
\end{enumerate}
\end{theorem}

\begin{remark}
 Note that  conditions  (\ref{eq:nopolygon1}) and (\ref{eq:nopolygon2}) are complementary whenever  $(B^2)\geq 5$ is met. 
\end{remark}

\begin{remark}
In his seminal work \cite{D} Demailly introduced  Seshadri constants with the aim of controlling the asymptotic growth of separation of jets by an ample line bundle 
at the given point. Our  Theorem~\ref{thm:nopolygon}.(i) can be viewed  as a more effective version of Demailly's  idea as explained in 
\cite{PAGI}*{Theorem~5.1.17} and\cite{PAGI}*{Proposition~5.1.19}.

To see this  note that  $\epsilon(B;x)$ equals  the largest  $\lambda$ such that the   inverted simplex $\Delta_{\lambda}^{-1}$ is contained  in $\Delta_{(E,z)}(\pi^*B)$ 
for any $z\in E$ by Theorem~\ref{thm:Seshadri inverted simplex}. 
\end{remark}

Coupled with Nadel vanishing we obtain the following effective  global generation result, reminiscent of  Reider's theorem in the spirit of  Demailly's original 
line of thought\footnote{Corollary~\ref{cor:eff glob gen} also follows from Theorem~\ref{thm:nopolygon} and \cite{EL3}*{Proposition 1.4}, but the proof of the latter 
is more involved.}.

\begin{corollary}\label{cor:eff glob gen}
Let $X$ be a smooth projective surface, $x\in X$, $L$ an ample line bundle on $X$ with $(L^2)\geq 5$. 
If  $\epsilon(L;x)\geq \tfrac{5-\sqrt{5}}{2}$, then 
$x$ is not a base point of the adjoint linear series $|K_X+L|$.
\end{corollary}

\begin{proof}[Proof of Theorem~\ref{thm:nopolygon}]
$(i)$  Let us fix a point $z\in E$, hence an infinitesimal flag $(E,z)$. We  show first that for any $0<c\ll 1$  condition $(\ref{eq:nopolygon1})$ implies 
the existence of a $\QQ$-effective divisor $D'\equiv (1-c)\pi^*B$ with $\sJ(X',D')|_U=\sO_U(-2E)$ for some open neighborhood $U$ of the exceptional divisor. 

Note that  $(\ref{eq:nopolygon1})$ is an open condition hence it is also satisfied 
for the divisor class $\pi^*((1-c)B)$ whenever $0<c\ll 1$. In what follows fix a rational number $c>0$ such that the above property holds, and set  $B'=\pi^*((1-c)B)-2E$. 
Condition $(\ref{eq:nopolygon1})$ yields that $\Delta_{(E,z)}(\pi^*B)$ contains an interior point $(t,y)\in\Lambda$ with $2\leq t<\mu_E(\pi^*B)$, therefore $\pi^*B-2E$ is 
a big $\QQ$-divisor on $X'$.  Also, by \cite{KL}*{Remark 1.7}  we know that 
\[
\Delta_{(E,z)}(B') \equ \Big(\Delta_{(E,z)}(\pi^*((1-c)B)) \ \cap \ \{(t,y) \ | \ t\geq 2\}\Big) \ -\ 2\eone\ ,
\]
where $\eone$ stands for the point $(1,0)$. 

Write $B'=P+N$ for  the   Zariski decomposition of $B'$; we will look for  the divisor $D'$ in the form
\[
D' \equ  P' + N +2E \ (\equiv \pi^*((1-c)B))\ ,
\]
with $P'\equiv P$ is an effective divisor\footnote{Note that the above expression is in general not the Zariski decomposition
of $D'$}. 

Since $P$ is big and nef, \cite{PAGI}*{Theorem 2.3.9} shows that 
one can find an effective divisor $N'$ and a sequence of ample $\QQ$-divisors $A_k$ with the property that  
$P=A_k+\tfrac{1}{k}N'$ for  $k\gg 0$. Choose $P'\equiv P$ to be an effective $\QQ$-divisor such that $A_k=P'-\tfrac{1}{k}N'$ is general and effective. This yields 
\[
\sJ(X',D') \equ  \sJ(X',A_k+\frac{1}{k}N'+N+2E) \equ  \sJ(X', \frac{1}{k}N'+N+2E) \equ  \sJ(X',N)\otimes \sO_{X'}(-2E)\ ,
\]
for $k\gg 0$. The second equality is an application of the Koll\'ar--Bertini theorem   \cite{PAGII}*{Example~9.2.29}, the third one 
comes from invariance under small perturbations (see \cite{PAGII}*{Example~9.2.30}) and  \cite{PAGII}*{Proposition 9.2.31}. 

Thus, it remains to check that  $\sJ(X',N)$ is trivial  at any point $z\in E$.  Since 
\[
\Delta_{(E,z)}(B') \ \bigcap \ \{0\}\times [0,1)\ \neq \ \varnothing,\ \forall z\in E\ ,
\]
\cite{LM}*{Theorem 6.4} implies 
\[
1 \ > \ \ord_z(N|_E), \forall z\in E\ .
\]
On the other hand,  $\ord_z(N|E)\geq \ord_z(N)$  yields $\ord_z(N) < 1$, for all $z\in E$. In the light of  \cite{PAGII}*{Proposition 9.5.13} this implies that $\sJ(X',N)$ is trivial at any $z\in E$, as  needed.

By  Lemma~\ref{lem:going down} there exists an effective $\QQ$-divisor $D\equiv (1-c)B$ on $X$ with $D'=\pi^*D$.  Hence the birational transformation rule for multiplier ideals (see \cite{PAGII}*{Theorem~9.2.33}) we have the sequence of 
equalities 
\[
\sJ(X,D) \equ \pi_*\big(\sO_{X'}(K_{X'/X})\otimes \sJ(X',D')\big) \equ \pi_*(\sJ(X',N))\otimes \sI_x
\]
where $K_{X'/X}=E$. Since $\sJ(X',N)$ is trivial at any $z\in E$, this means that  $\sJ(X,D)=\sI_x$ in a neighborhood of the point $x\in X$.

\noindent
$(ii)$  Set $\epsilon\deq \epsilon(B;x)$ and  observe that  \cite{KL}*{Theorem D} yields the containment
\[
\triangle OAA'  \dsubseteq \Delta_{(E,z_0)}(\pi^*B), \text{ where } O\equ (0,0), A\equ (\epsilon,0), \text{ and } A'\equ (\epsilon,\epsilon)\ .
\]
If $\epsilon>2$, then 
\[
\intt \big(\Delta_{(E,z_0)}(\pi^*B)\ \bigcap \ \Lambda\big) \dsupseteq \intt \big(\triangle OAA'\cap \Lambda\big) \neq \varnothing
\]
contradicting condition (\ref{eq:nopolygon2}). The case $\epsilon=2$ is equally impossible since that would imply via  (\ref{eq:nopolygon2}) that 
$\Delta_{(E,z_0)}(\pi^*B)$ lies to the left of the  line $t=2$;   as it lies underneath the diagonal anyway, it would have volume at most $2$, 
contradicting $(B^2)\geq 5$. Therefore we can safely assume $\epsilon<2$. 

Now, condition~$(\ref{eq:nopolygon2})$ implies that the segment $\{2\}\times [0,1)$ does not intersect  $\Delta_{(E,z_0)}(\pi^*B)$. Since the latter is convex, it must lie above some  line $\ell$ that passes through the point $(2,1)$ and below the diagonal. 
Let $(\delta,0)$ be the point of intersection of  $\ell$ and  the $t$-axis. It follows  that $\delta\geq \epsilon$, since we know that the inverted simplex $\triangle OAA'$ is contained in $\Delta_{(E,z_0)}(\pi^*B)$. So, our  goal is now to find 
an upper bound on $\delta$. 

We can assume that both $\epsilon,\delta>1$, since $\epsilon\leq 1$ already implies our statement. In this case we have $\tfrac{1}{2-\delta}>1$ for the slope of the line $\ell$, therefore the diagonal and $\ell$ intersect at the  point  
$(\tfrac{\delta}{\delta-1},\tfrac{\delta}{\delta-1})$ in the first quadrant. The triangle formed by $\ell$, the diagonal,  and the $t$-axis contains our 
Newton--Okounkov polygon, therefore its area is no smaller than the area of $\Delta_{(E,z_0)}(\pi^*B)\geq 5/2$. 
Hence, we obtain $\epsilon(B;x)\leq\delta\leq \frac{5-\sqrt{5}}{2}$. 

\end{proof}

\begin{lemma}\label{lem:going down}
 Let $Y$ be a smooth projective variety, $Z\subseteq X$ a smooth subvariety, $\pi\colon Y'\to Y$ the blowing-up of $Y$ along $Z$. Furthermore, let 
 $B$ be a Cartier divisor on $Y$, $D'$ a Cartier divisor on $Y'$. 
 If $D'\equiv \pi^*B$, then there exists a divisor $D\equiv B$ on $Y$ such that $D'\equ \pi^*D$. 
\end{lemma}
  
\begin{proof} This follows from  the fact that $F'\equiv 0$ for a Cartier divisor $F'$ on $Y'$ (with integral, $\QQ$-, or $\RR$-coefficients) implies $F'=\pi^*F$ 
for a numerically trivial divisor $F$ on $Y$. Indeed, apply this with $F'\deq D'-\pi^*B$ to obtain a numerically 
trivial divisor $F$ on $Y$ with $\pi^*F=D'-\pi^*B$, and set $D\deq B+F$. 
\end{proof}


\subsection{Singular divisors at very general points.}
Ever since the birth of the concept it has been an important guiding principle that local positivity of a line bundle is considerably easier to control
at a general or a very general point. This observation is manifestly present in the work of Ein--K\"uchle--Lazarsfeld \cite{EKL} (see also \cite{EL}),
where the authors  give a lower bound on Seshadri constants at very general points depending only on the dimension of the ambient space. 

Later, Nakamaye \cite{NakVeryGen} elaborated some of the ideas of \cite{EKL}, while translating them to the language of multiplicities. This thread was in turn
picked up in \cite{KL}, and further developed in the framework of infinitesimal Newton--Okounkov bodies of surfaces, as seen in Proposition~\ref{prop:genericinf}. It is hence not surprising that 
one can expect stronger-than-usual results on singular divisors at very general points.

\begin{theorem}\label{thm:verygeneric}
Let $p$ be a positive integer, $X$ a smooth projective surface, $L$ an ample line bundle on $X$ with $(L^2)\geq 5(p+2)^2$. 
Let  $x\in X$ be  a very general point, and assume  that there is no irreducible curve $C\subseteq X$ smooth at $x$ with $1\leq (L\cdot C)\leq p+2$. 

Then, for some (or,  equivalently, every) point $z\in E$
\begin{equation}\label{eq:verygeneric}
\length\Big(\Delta_{(E,z)}(\pi^*B)\cap \{2\}\times \RR\Big) \ > \ 1\ , 
\end{equation}
where as usual we write  $B\deq\frac{1}{p+2}L$.
\end{theorem}

\begin{corollary}\label{cor:very general}
Under the assumptions of  Theorem~\ref{thm:verygeneric}, there always exists an effective $\QQ$-divisor $D\equiv (1-c)B$ for some $0<c\ll 1$
such  that $\sJ(X;D)=\sI_{X,x}$ in a neighborhood of the point $x$.
\end{corollary}
\begin{proof}
By Proposition~\ref{prop:genericinf} any infinitesimal Newton--Okounkov polygon sits under the diagonal in $\RR^2$. Therefore $(\ref{eq:verygeneric})$ implies that condition $(\ref{eq:nopolygon1})$ from Theorem~\ref{thm:nopolygon} is satisfied, hence the claim. 
\end{proof}

\begin{remark}
As abelian surfaces are homogeneous, Theorem~\ref{thm:verygeneric} holds for all points on them. 
\end{remark}

Before proceeding with the proof, we make some preparations. Let $\epsilon=\epsilon(L;x)$, and write  $\epsilon_1 \deq \e(B;x)=\frac{1}{p+2}\epsilon$. 
As  Zariski decompositions of the divisors $\pi^*B-tE$ along the line segment $t\in [\epsilon_1,2)$ will play a decisive role, we will fix notation
for them as well. 

By \cite{KLM}*{Proposition 2.1}  there exist only finitely many curves $\og_1,\dots,\og_r$ that occur in the negative part of $\pi^*B-tE$ 
for any $t\in [\epsilon_1,2)$. Write  $\e_i$ for  the value of $t$ where $\og_i$ first appears. We can obviously assume
that $\e_1\leq\e_2\leq \ldots \leq \e_r$, in addition we put $\e_{r+1}=2$. Denote also by $m_i\deq (\og_i\cdot E) \equ\mult_x(\Gamma_i)$ for $1\leq i\leq r$, where $\Gamma_i=\pi(\og_i)$. 
For any $t\in[\epsilon_1 ,2)$ let $\pi^*B-tE=P_t+N_{\pi^*(B)-tE}$ be the Zariski decomposition of the divisor.

\begin{lemma}\label{lem:Nakamaye and Zariski}
With notation as above, we have 
\[
 N_{\pi^*B-tE} \equ \sum_{i=1}^{r}\mult_{\og_i}(\| \pi^*B-tE\|) \og_i
\]
for all $t\in [\e_1,2)$. 
\end{lemma}
\begin{proof}
It follows from the definition of asymptotic multiplicity and the existence and uniqueness of Zariski decomposition that for an arbitrary big and 
nef $\RR$-divisor $D$ and an irreducible curve $C$, the expression $\mult_C \|D\|$ picks up the coefficient of $C$ in the negative part of $D$.    
\end{proof}

\begin{lemma}\label{lem:estimate for many curves}
With notation as above, if $x\in X$ is a very general point, then
\[
\length\Big(\Delta_{(E,z)}(\pi^*B)\cap\{t\}\times\RR\Big) \equ  (P_t\cdot E) \dleq l_i(t) \,\deq\, (1-\sum_{j=1}^{i}m_j)t + \sum_{j=1}^{i}\e_jm_j \ ,
\]
for all  $t\in [\epsilon_i,\epsilon_{i+1}]$ and all  $z\in E$.
\end{lemma}
\begin{remark}
Lemma~\ref{lem:estimate for many curves} is essentially a restatement of \cite{NakVeryGen}*{Lemma 1.3} in the context of Newton--Okounkov bodies. Note that Nakamaye's 
claim is strongly based on \cite{EKL}*{\S 2} (see also  \cite{PAGI}*{Proposition~5.2.13}), providing a way 
of  ``smoothing divisors in affine families'' using differential operators. This observation will play  an important role in the proof of Theorem~\ref{thm:verygeneric} 
and consequently in that of  Theorem~\ref{thm:np}.
\end{remark}
 
Returning to the verification of  Lemma~\ref{lem:estimate for many curves}, let $A_i$ stand for  the point of intersection of the lines $t=\e_i$ and 
$l_i$ (here we identify the function $l_i$ with its graph). Note that the function $\ell\colon [\e_1,2]\rightarrow \RR_+$ given by $\ell(t) \deq l_i(t)$ on each interval $[\e_i,\e_{i+1}]$ is continuous, 
and satisfies $(\e_i,\ell(\e_i))=A_i$. Let $T_i$ be the polygon spanned by the origin, the points $A_1,\dots,A_i$, and the intersection point $F_i$  of the 
horizontal axis with the line $l_i$. 

\begin{remark}\label{rmk:estimate}
Since the upper boundary of the polygon $\Delta_{x}(B)$ is concave, we have 
\[
 \Delta_x(B) \ \dsubseteq \ T_i, \text{ for all } 1\leq i\leq r \ . 
\]
An explicit computation using the definition will convince that the subsequent slopes of the functions $l_i$  are decreasing, therefore 
one  has $T_i \dsupseteq T_{i+1}$ for all $1\leq i\leq r$ as well.
\end{remark}

\begin{proof}[Proof of Lemma~\ref{lem:estimate for many curves}]
It follows from \cite{KL}*{Remark 1.9} that 
\[
\length\Big(\Delta_{(E,z)}(\pi^*B)\cap\{t\}\times\RR\Big) \equ  (P_t\cdot E), \textup{ for all points } z\in E \ .
\]
The point  $x\in X$ was chosen to be very general, therefore one has  $\mult_{\overline{\Gamma}_i}(\|\pi^*(B)-t\og_i\|) \ \geq \ t-\epsilon_i$ 
for all $1\leq i\leq r$ and all $t\geq \e_i$ by Nakamaye's Lemma \cite{KL}*{Lemma 4.1} (see also \cite{NakVeryGen}*{Lemma 1.3}). As a consequence, 
we obtain via Lemma~\ref{lem:Nakamaye and Zariski} that 
\begin{eqnarray*}
 (P_t\cdot E) & = & \big( \big( (\pi^*B-tE) - \sum_{i=1}^{r}\mult_{\og_i}(\| \pi^*B-tE\|) \og_i \big) \cdot E\big) \\ 
 & \leq &  t - \sum_{j=1}^{i}(t-\e_j)m_j \equ (1-\sum_{j=1}^{i}m_j)t + \sum_{j=1}^{i}\e_jm_j
\end{eqnarray*}
 for all $\e_i\leq t< 2$, as required.
\end{proof}

\begin{proof}[Proof of Theorem~\ref{thm:verygeneric}]
We will subdivide the proof  into  several cases   depending on the size and nature  of the Seshadri constant $\epsilon\deq \epsilon(L;x)$.

\smallskip\noindent
\textit{Case 1:} $\epsilon(L;x)\geq 2(p+2)$. Observe that \cite{KL}*{Theorem D} yields
\[
\length \big(\Delta_{(E,z)}(\pi^*B)\cap \{2\}\times\RR\big) \ > \ 1 \ ,
\]
since $\epsilon_1=\epsilon(B;x)=\frac{1}{p+2}\epsilon\geq 2$,  and  $\Delta_{(E,z)}(\pi^*B) \supseteq \Delta^{-1}_{\epsilon(B;x)} \supseteq \Delta^{-1}_2$.

\smallskip\noindent
\textit{Case 2:} $\epsilon(L;x)=\frac{(L.F)}{\textup{mult}_x(F)}$ for some irreducible curve $F\subseteq X$ with 
$r\deq (L.F)\geq q\deq\mult_x(F)\geq 2$ (note that $r\geq q$ follows from the main result of \cite{EL}).

The point  $x\in X$ was chosen to be  very general, therefore  Proposition~\ref{prop:genericinf} implies  $\Delta_x(B)\subseteq \triangle OMN$, where 
$O=(0,0),M=(\frac{r}{(p+2)q},\frac{r}{(p+2)q)})$ and $N=(\frac{r}{(p+2)(q-1)},0)$. Since the area of the first polygon is $B^2/2\geq 5/2$, we obtain the inequality 
\[
  \frac{1}{(p+2)^2}\cdot \frac{r^2}{q(q-1)} \dgeq 5\ .
\]
Remembering that $q\geq 2$, we arrive at  the following lower bound on the Seshadri constant
\begin{equation}\label{eqn:6}
 \epsilon(B;x) \equ \frac{r}{(p+2)q} \dgeq \sqrt{5\zj{1-\frac{1}{q}}} \dgeq \sqrt{\frac{5}{2}}\ ,
\end{equation}
Next, fix a point $z_0\in E$, and  let $S=(\epsilon_1,\epsilon_1)$, $T=(\epsilon_1,0)$, $S'=(2,2)$ and $T'=(2,0)$. The triangle $\triangle OST$ is the largest inverted standard simplex  inside  $\Delta_{(E,z_0)}(\pi^*B)$. Let 
\[
[AD] \deq \Delta_{(E,z_0)}(\pi^*B)\cap \{2\}\times\RR  \ ,
\]
and aiming at a contradiction, suppose  $||AD||\leq 1$ (for a visual guiding see Figure~\ref{fig:2}).

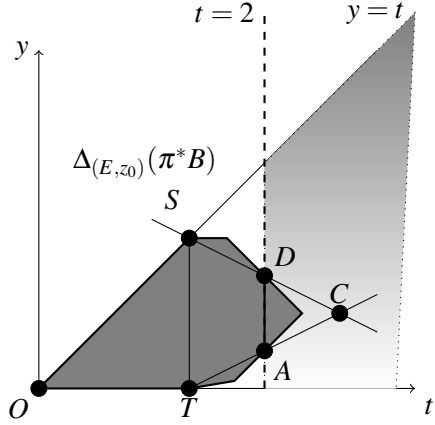
\begin{figure}
\begin{tikzpicture}
%
%
\draw [->]  (0,0) -- (5,0);   
\node [below right] at (5,0) {$t$};
\draw [->]  (0,0) -- (0,4.5); 
\node [left] at (0,4.5) {$y$}; 
\draw (0,0) -- (4.5,4.5); 
\node [left] at (5,5) {$y=t$};
\shadedraw [dotted,shading=axis] (3,0) -- (3,3) -- (5,5) -- (4.75,0);
\draw [fill=gray, thick] (0,0) -- (2,2) -- (2.5,2) -- (3,1.5) -- (3.5,1) -- (3,0.5) -- (2.6,0.1) -- (2,0) -- (0,0);  
\node [left] at (2.5,3) {$\Delta_{(E,z_0)}(\pi^*B)$}; 
\draw [thick, dashed] (3,0) -- (3,5);  
\node [left] at (3,5) {$t=2$}; 
\draw (2,0) -- (2,2); 
\draw [thick] (3,0.5) -- (3,1.5);  
\draw (1.5,2.25) -- (4.5,0.75); 
\draw (2,0) -- (4.5, 1.25); 
%
%
\node [below left] at (0,0) {$O$}; 
\filldraw[black]  (0,0) circle (3pt);
\node [below] at (2,0) {$T$}; 
\filldraw[black] (2,0) circle (3pt);
\node [above left] at (2,2.25)  {$S$}; 
\filldraw[black] (2,2) circle (3pt); 
\node [above right] at (3,1.5) {$D$}; 
\filldraw [black]  (3,1.5) circle (3pt); 
\node [below  right] at (3,0.5) {$A$}; 
\filldraw[black]  (3,0.5) circle (3pt);
\node [above] at (4,1) {$C$}; 
\filldraw[black] (4,1) circle (3pt); 
\end{tikzpicture}
\caption{$\Delta_{(E,z_0)}(\pi^*B)$ and the triangles $\triangle STC$ and $\triangle ADC$}\label{fig:2}
\end{figure}  

Note first that by $(\ref{eqn:6})$, one has $\epsilon_1>1$ and thus $||ST||>1$. Together with the assumption  $||AD||\leq 1$,  this implies that the lines $TA$ and $SD$ 
intersect to the right of  the vertical line $t=2$, let us call the point of intersection  $C$. Denote by $x=\dist(C, t=2)$. Since both line segments $[TA]$ and $[SD]$ 
are contained in $\Delta_{(E,z_0)}(\pi^*B)$, convexity yields the inclusion
\[
\triangle ADC \ \supseteq \ \Delta_{(E,z_0)}(\pi^*B)_{t\geq 2} \deq  \Delta_{(E,z_0)}(\pi^*B) \ \cap \ \{t\geq 2\}\times\RR \ .
\]
An area comparison yields the  string of inequalities
\begin{align*}
\Area(\triangle ADC) \ &\geq \ \Area(\Delta_{(E,z_0)}(\pi^*(B))_{t\geq 2} \equ   \Area(\Delta_{(E,z_0)}(\pi^*(B)) \ - \ \Area(\Delta_{(E,z_0)}(\pi^*(B))_{t\leq  2} \\
& \geq \frac{\vol_X(B)}{2} -\Area(\triangle OS'T') \dgeq  \frac{5}{2}-\frac{4}{2} \equ  \frac{1}{2} \ .
\end{align*}
By  the similarity between $\triangle ADC$ and $\triangle TSC$, we see that  
\[
||AD|| \equ  \frac{\epsilon_1 x}{x+2-\epsilon_1}\ .
\]
Along  with the condition  $||AD||\leq 1$ this  implies
\begin{equation}\label{eqn:7}
\frac{x+2}{x+1} \ \geq \ \epsilon_1 \ .
\end{equation}
On the other hand, by the above we also have 
\[
\Area(\triangle ADC) \equ  \frac{\epsilon_1 x^2}{2(x+2-\epsilon_1)} \ \geq \ \frac{1}{2} \ ,
\]
which gives 
\[
\epsilon_1  \dgeq  \frac{x+2}{x^2+1} \ .
\]
Combining this inequality with $(\ref{eqn:7})$, we arrive at 
\[
\frac{x+2}{x+1} \dgeq  \frac{x+2}{x^2+1} \ ,
\]
forcing $x\geq 1$. Since the function $f(x)=\frac{x+2}{x+1}$ is decreasing for $x\geq 1$, the inequality $(\ref{eqn:7})$ implies that $\epsilon_1 \leq f(1)=3/2$, 
contradicting   $\epsilon_1\geq \sqrt{5/2}$ from  $(\ref{eqn:6})$. Thus $||AD||>1$, as required. 

\smallskip\noindent
\textit{Case 3:} Here we assume that $\epsilon=\epsilon(L;x)\in \NN$ with $1\leq \epsilon\leq 2p+3$, and that there exists an irreducible curve $\Gamma_1\subseteq X$ with $\mult_x(\Gamma_1)=1$ and $\epsilon(L;x)=(L\cdot \Gamma_1)=\epsilon$. Denote by $\og_1$ the proper transform via $\pi$ of $\Gamma_1$.

We point out that when  $1\leq \epsilon\leq p+2$, our assumptions  reduces to the condition in the statement of our theorem with  $\Gamma_1$ playing the role of  $C$.
Whence, in what follows  we can assume that in addition $p+3\leq \epsilon\leq 2p+3$. Secondly, since $\mult_x(\Gamma_1)=1$, $(L\cdot \Gamma_1)=\epsilon\leq 2p+3$,
and $L^2\geq 5(p+2)^2$, the Hodge Index Theorem yields  $(\Gamma_1^2)\leq 0$.

Since the point $x\in X$ is  very general, Proposition~\ref{prop:genericinf} (2.b) shows  that the slope of the linear function $t\longmapsto (P_t.E)$ is non-positive. 
By $(\ref{eq:positivepart})$,  the only way  this can happen is  when $m_1=1$ and $(\Gamma_1^2)=0$ (since  $(\Gamma_1^2)\leq 0$).

We consider two sub-cases.

\smallskip\noindent
\textit{Case 3(a):} $\og_1$  is the only curve appearing in the negative part of the divisor $\pi^*B-tE$ for any $t\in [\epsilon_1,2)$. 

Note first that by above we have that $\mult_x(\Gamma_1)=1$ and thus $\og_1=\pi^*\Gamma_1-E$. Since the curve $\og_1$ is 
the only curve appearing in the negative part of $\pi^*B-tE$ for any $t\in [1,2)$,  we can apply the algorithm for finding the negative part of the Zariski decomposition 
for each divisor $\pi^*B-tE$ and deduce that
\[
\pi^*B-tE \equ  P_t \ + \ (t-\epsilon_1)\overline{\Gamma}_1
\]
is indeed the appropriate Zariski decomposition for  any $t\in [\epsilon_1,2)$. In particular, we obtain that  
\begin{equation}\label{eq:positivepart}
(P_t.E) \equ  \epsilon_1\ \ \ \text{for all $t\in [\epsilon_1,2)$.}
\end{equation}

 Having positive self-intersection  $\pi^*B-tE$ is big for all $t\in [\epsilon_1,2]$, therefore the Zariski decomposition along this  line segment 
is continuous by \cite{BKS}*{Proposition 1.16}.  Accordingly,  $(\ref{eq:positivepart})$ yields 
\[
\length\big(\Delta_{(E,z)}(\pi^*B)\cap \{2\}\times\RR\big) \ \equ \ (P_2.E) \ \equ \ \epsilon_1 \ \equ \ \frac{\epsilon}{p+2} \ > \ 1 
\]
for  we assumed that $\epsilon\in\st{p+3,\ldots ,2p+3}$. 

\smallskip\noindent
\textit{Case: 3(b)} Assume  that the negative part of $\pi^*B-tE$ contains other  curve(s)  beside  $\overline{\Gamma}_1$ for some  $t\in [\epsilon_1,2)$. 
\\

Denote these other curves that appear in the negative part of the segment line $\pi^*B-tE$ for $t\in [\epsilon, 2)$ by $\og_2,\ldots ,\og_r$ for some $r\geq 2$. Our main tool is going to be  the generic infinitesimal Newton--Okounkov polygon $\Delta_{x}(B)$. An important property of $\Delta_x(B)$  is that the vertical line segment $\Delta_{x}(B)\cap \{t\}\times\RR$ starts on the $t$-axis for any $t\geq 0$ (see \cite{KL}*{Theorem 3.1}).
 
Aiming at a contradiction, suppose  that $\length(\Delta_{x}(B)\cap\{2\}\times\RR)\leq 1$. Note first that  this is equivalent to having the point $(2,1)$ outside the interior of the polygon $\Delta_{x}(B)$, as its lower boundary  sits on the $t$-axis.

As  $x\in X$ was chosen to be a very general point, by Remark~\ref{rmk:estimate} we have $\Delta_{x}(B) \subseteq T_2$, where the latter polygon denotes the convex hull  
of  the vertices  $O,A_1,A_2,F_2$, where   $O=(0,0)$ and $A_1=(\epsilon_1,\epsilon_1)$. We will focus on  the slope of the line $A_2F_2$. 
If  $A_2=A_1$  then by Lemma~\ref{lem:estimate for many curves}, we know that 
\[
 \text{slope of}\ A_2F_2 \dleq \text{slope of }\ \ell (t) \equ (1-m_1-m_2)t+\epsilon_1m_1+\epsilon_2m_2\ .
\]
In particular, the  slope of $A_2F_2$ is at most  $-1$, since $m_1=1$ and $m_2\geq 1$.

In the non-degenerate case $A_1\neq A_2$, since we have $m_1=1$, then 
\[
 \text{slope of }\ A_2F_2 \dleq \text{slope of }\ \ell (t) \equ -m_2t+\epsilon_1+\epsilon_2m_2 \dleq -1
\]
again. Note also that in this case, based on the proof of $Case~3(a)$, we know for sure that the segment $[A_1A_2]$ is actually an edge of the convex polygon $\Delta_{x}(B)$.

Since  the upper boundary of the polygon $\Delta_x(B)$ is concave by \cite{KLM}*{Theorem B}, then  all the supporting lines of the edges on this boundary, besides $[OA_1]$ and 
$[A_1A_2]$, have slope at most  $-1$. However, we initially assumed  that $(2,1)\notin\intt\Delta_x(B)$, thus convexity yields  that the first edge of the polygon 
$\Delta_{x}(B)$ intersecting  the region $[2,\infty)\times\RR$  will  do so at a point on the line segment
  $\{2\}\times[0,1]$. Furthermore, this edge and all the other edges of the upper boundary in the region $[2,\infty)\times \RR$ will have slope at most  $-1$. 

Now, set  $A=(2,1), C=(3,0)$ and $D=(2,0)$. Since the line $AC$ has slope $-1$, then by what we said just above, we have the following inclusions due to convexity reasons
\[
\Delta_x(B)\cap [2,\infty)\times\RR  \dsubseteq  \triangle ADC\ .
\]
Write  $R=(2,2)$; since  $\Area(\Delta_x(B))\geq 5/2$, the above  inclusion implies 
\[
\Area(\triangle ADC) \equ  \frac{1}{2} \ > \ \Area(\Delta_x(B)) - \Area(\triangle ODR)  \geq \frac{5}{2}-2 \equ  \frac{1}{2}
\]
contradicting $(2,1)\notin\intt \Delta_{x}(B)$.  In particular, $\length (\Delta_x(B)\cap\{2\}\times\RR)>1$,  and we are done. 
\end{proof}

\section{Syzygies of abelian surfaces}

\subsection{Syzygies and singular divisors on abelian surfaces}
The essential contribution of the work \cite{LPP} can be summarized in the following statement. 

\begin{theorem}(\cite{LPP})\label{thm:LPP}
Let $X$ be an abelian surface, $L$ an ample line bundle on $X$, and $p$ a positive integer such that there exists an effective $\QQ$-divisor $F_0$ on $X$ satisfying
\begin{enumerate}
 \item $F_0 \equiv \frac{1-c}{p+2} L$ for some $0<c\ll 1$, and 
 \item $\sJ(X;F_0)\equ \sI_0$, the maximal ideal at the origin. 
\end{enumerate}
Then $L$ satisfies the property $(N_p)$. 
\end{theorem}

For the sake of clarity we give a quick outline of the argument in \cite{LPP}; to this end,  we  quickly recall some terminology.  As above, $p$ will denote a natural number, 
one  studies sheaves on the $(p+2)$-fold self-product $X^{\times (p+2)}$. Based on the philosophy of Green \cite{Green2}*{\S 3} and later established by Inamdar in \cite{I}, the $N_p$ property for the line bundle  $L$ holds provided 
\begin{equation}\label{eqn:Green}
H^i(X^{\times (p+2)},\boxtimes^{p+2}L\otimes N\otimes \sI_{\Sigma}) \equ 0\ \ \text{for all $i>0$}, 
\end{equation}
and for any nef line bundle $N$, where $\sI_{\Sigma}$ is the ideal sheaf of the reduced algebraic subset
\[
 \Sigma \deq \st{(x_0,\dots,x_{p+1})\mid x_0=x_i\ \text{for some $1\leq i\leq p+1$}}\ .
\]
This is the vanishing condition one  intends to verify. 

Observe that $\Sigma\subseteq X^{\times (p+2)}$ can be realized in the following manner. Upon forming the self-product $Y\deq X^{\times (p+1)}$ with projection maps 
$\pr_i\colon Y\to X$, one considers  the subvariety
\[
 \Lambda \deq \bigcup_{i=1}^{p+1}\pr_i^{-1}(0) \equ \st{(x_1,\dots,x_{p+1})\mid x_i=0\ \ \text{for some $1\leq i\leq p+1$}}\ .
\]
Next, look at the morphism 
\[
 \delta\colon X^{\times (p+2)} \to  X^{\times (p+1)}\ ,\ (x_0,\dots,x_{p+1}) \mapsto (x_0-x_1,\dots,x_0-x_{p+1})\ ,
\]
then $\Sigma \equ \delta^{-1}(\Lambda)$ scheme-theoretically. Consider the divisors 
\[
 E_0 \deq  \sum_{i=1}^{p+1} \pr_i^*F_0\ \ \text{ and }\ \ E\deq \delta^*E_0\ ,
\]
as forming multiplier ideals commutes with taking pullbacks by smooth morphisms (cf. \cite{PAGII}*{9.5.45}), one observes that 
\[
 \sJ(X^{\times p+2};E) \equ \sJ(X^{\times p+2};\delta^*E_0) \equ \delta^*\sJ(X^{\times p+1};E_0) \equ \delta^*\sI_{\Lambda} \equ \sI_{\Sigma}\ .
\]
The divisor $(\boxtimes_{i=1}^{p+2}L)-E$ is ample by \cite{LPP}*{Proposition 1.3}, therefore 
\[
 H^i(X^{\times (p+2)},\boxtimes^{p+2}L\otimes N\otimes \sI_{\Sigma}) \equ H^i(X^{\times (p+2)},\boxtimes^{p+2}L\otimes \sJ(X^{\times p+2};E)) \equ 0 
 \ \text{ for all $i>0$}
\]
by Nadel vanishing, where one sets $N=0$.

\subsection{Projective normality on abelian surfaces.}

As a toy example, we provide an easy proof of Theorem~\ref{thm:np} for projective normality via the convex geometric results given in Theorem~\ref{thm:nopolygon} and 
Theorem~\ref{thm:verygeneric} $(2)$. What sets this apart from the general case is that it suffices  to find a divisor whose multiplier ideal is the 
maximal ideal of the origin in a small neighborhood.  This is of course  much simpler to check than the global statement of  \cite{LPP}.
Historically, this was the first case that was dealt with by Hwang and To \cite{HT}. Observe that our methods provide   a precise characterization of projective normality in terms of existence of elliptic 
curves with given numerical properties. 

\begin{corollary}\label{cor:normality}
Let $X$ be an abelian surface and $L$ an ample line bundle with $(L^2)\geq 20$. Then the following  are equivalent:
\begin{enumerate}
\item $X$ does not contain any elliptic curve $C$ with $(L\cdot C)\leq 2$ and $(C^2)=0$.
\item $L$ defines a projectively normal embedding.
\end{enumerate}
\end{corollary}
\begin{remark}\label{rem:veryample vs normality}
As remarked in the introduction, there seems to be a discrepancy between Corollary~\ref{cor:normality} and Reider's theorem from \cite{R}. In the case of checking very ampleness condition for an ample line bundle $L$ on an abelian surface $X$, the latter says that if $L^2\geq 10$, then $L$ is very ample on $X$ if and only if there is no elliptic curve $C\subseteq X$ with $C^2=0$ and $1\leq (L.C)\leq 2$. So, now consider  a general abelian surface $X$ given by a principal polarization $L$ of type $(1,6)$ and hence $L^2=12$. By Reider's theorem  then $L$ is very ample. In order for $L$ to be projectively normal,
the map $\textup{Sym}^2(H^0(X, L))\rightarrow H^0(X, L^{\otimes 2})$ needs to be  surjective. By  Riemann--Roch,  the dimension of the first vector space is $21$, while  that of the image is $24$, therefore  $L$ is very ample but not projectively normal. 

A similar  observation holds for the cases  $(1,5)$, $(2,2)$; note that in the second case the line bundle is not  very ample but it still defines a $2:1$ morphism of $X$ to the Kummer surface. For the cases $(2,d)$ with $d\geq 4$,  a result of Pareschi and Popa  \cite{PP}*{Theorem 6.1} yields  that $L$ is in fact  projectively normal when $X$ is a generic choice of abelian surface with the given polarization. It turns out that it is much harder to say when a very ample line bundle is not projectively normal, 
especially that the only difference by Reider's theorem and Corollary~\ref{cor:normality} in these two cases is the lower bound on  $(L^2)$.

\end{remark}

\begin{proof}
$(1)\Rightarrow (2)$ The main observation is  that in order to prove projective normality on an abelian surface, it is enough to show $\sJ(X;D)\equ \sI_o$ 
in a neighborhood $U\subseteq X$ of the origin  for some effective $\QQ$-divisor 
\[
D \ \equiv \ (1-c)B \ \deq \ \frac{1}{2}(1-c)L \ .
\]
To see this consider the map $\sigma:X\times X\rightarrow X$ given by  $\sigma(x,y)=x-y$.  Then  
\[
\sJ(X\times X,\sigma^*(D)) \equ  \sigma^*(\sJ(X;D))= I_{\Delta}\cdot \sI\ ,
\]
where $\Delta$ stands for  the diagonal, and $\sI\subseteq \sO_{X\times X}$ is an ideal sheaf whose cosupport $V$ is contained in $\sigma^{-1}(X\setminus U)$, 
hence  disjoint from $\Delta$. Now, one has the short exact sequences  
\[
0\rightarrow I_{\Delta}\cdot \sI\rightarrow
\sO_{X\times X}\rightarrow \sO_{X\times X}/(I_{\Delta}\cdot \sI)\rightarrow 0
\]
and
\[
0\rightarrow I_{\Delta}\cdot \sI\rightarrow
I_{\Delta}\rightarrow I_{\Delta}/(I_{\Delta}\cdot\sI)\rightarrow 0 \ ,
\]
since  $\Delta\cap V=\varnothing$, which also implies by straightforward commutative algebra that  $I_{\Delta}/(I_{\Delta}\cdot\sI)$ is a direct summand of 
$\sO_{X\times X}/(I_{\Delta}\cdot \sI)$. 

Therefore, in order  to prove the vanishing of higher cohomology of shifts of   $I_{\Delta}$, one shifts the first sequence and uses Nadel vanishing for the first two 
terms to obtain  vanishing for the cohomology of the shifts of the quotient. A second application of Nadel to the second sequence and the fact that taking 
cohomology commutes with direct sums gives the required vanishing. 
 
To finish off the proof we need to present  a divisor $D\equiv (1-c)B$ with  $\sJ(X;D)=\sI_{X,o}$, where the equality takes place in a neighborhood of the origin. Since we work on an abelian surface, 
 we can assume that $x\in X$ is a very general point. As in our previous proofs, we approach the question through an analysis of the Seshadri constant  $\epsilon(B;x)=\frac{1}{2}\epsilon(L;x)$.

If  $\epsilon(L;x)=r/q$ with $r\geq q\geq 2$, then Case $(2)$ of Theorem~\ref{thm:verygeneric} and Theorem~\ref{thm:nopolygon} imply  the existence of such a divisor. If  $\epsilon(L;x)\geq 4$ then again condition  $(\ref{eq:nopolygon1})$ in  Theorem~\ref{thm:nopolygon} applies, and we are done. 

Thus,  it remains to tackle the cases when there exists a curve $C\subseteq X$ with $\mult_{x}(C)=1$ and $(L\cdot C)=\epsilon(L;x)$ for which  $\epsilon(L;x)=1,2,3$. 
If  $\epsilon(L;x)=1,2$, then the  Hodge index theorem yields  $(C^2)\leq 0$, and  since $X$ is abelian, $(C^2)=0$. The adjunction formula then shows that 
$C$ must be  an elliptic curve as abelian varieties do not  contain rational curves. 

If  $\epsilon(L;x)=3$, then  $\epsilon(B;x)\geq 3/2$. Note that 
\[
\frac{3}{2} \dgeq  \frac{5-\sqrt{5}}{2} \ ,
\]
thus,  by  Theorem~\ref{thm:nopolygon} (ii) we deduce the existence of a divisor with the desired properties.

$(2)\Rightarrow (1)$ In the opposite direction, suppose  that $L$ satisfies property $N_0$, i.e. it is  projectively normal,  and that there exists an elliptic curve $C\subseteq X$ with $(L\cdot C)=1$ or $2$. 
The first condition implies that $L$ is a very ample divisor and defines an embedding $X\subseteq \PP(H^0(X,L))$. On the other hand, using \cite{PAGI}*{Theorem 2.2.15} and    the fact that $L^2\geq 20$, we know 
 that $L-C$ is big; then it  is automatically  nef since $X$ is  an abelian surface.  The  Kawamata-Viehweg vanishing theorem then implies  that the restriction map
\[
H^0(X,L) \ \rightarrow \ H^0(C,L|_C)
\]
is surjective. Since $L$ is very ample, $L|_C$ defines  an embedding of  $C$ into some projective space as well, however as  $1\leq (L\cdot C)\leq 2$, this can never be the case.
\end{proof}

\subsection{Property $(N_p)$ in the  absence of elliptic curves of low degree}

In this subsection we give a proof of the direct implication of Theorem~\ref{thm:np}. As opposed to the case of projective normality  it is no longer clear whether 
finding an effective divisor $D$ with $\sJ(X;D)=\sI_o$  locally suffices to verify $(N_p)$. We show however, that with a bit more work one can in fact control the multiplier 
ideal of the divisor found in Theorem~\ref{thm:nopolygon} over the whole abelian surface $X$. As always, $\pi\colon X'\to X$ denotes the blow-up of $o$ with exceptional 
divisor $E$. 

The main goal of this subsection is to prove the following theorem.

\begin{theorem}\label{thm:nopolygonabelian}
Let $X$ be an abelian surface and $B$ an ample $\QQ$-divisor on $X$ with $(B^2)>4$. Suppose that
\begin{equation}\label{eq:nopolygonabelian}
\textup{length}\big(\Delta_{(E,z_0)}(\pi^*(B)) \cap \{2\}\times\RR \big) \ > \ 1 \ ,
\end{equation}
for some point $z_0\in E$. Then there exists an effective $\QQ$-divisor $D\equiv (1-c)B$ for some $0<c<1$ such that $\sJ(X,D)=\sI_{X,o}$ over the whole of $X$.
\end{theorem}

\begin{proof}[Proof of Theorem~\ref{thm:np}, $(1)\Rightarrow (2)$]
By Theorem~\ref{thm:LPP} it suffices  to find a divisor as produced by Theorem~\ref{thm:nopolygonabelian}. Since $X$ is abelian, it is enough to treat the case 
when the origin $o$ behaves like a very general point. By Theorem~\ref{thm:verygeneric} the  condition $(\ref{eq:nopolygonabelian})$ is automatically satisfied whenever 
$X$ does not contain an elliptic curve $C$ with $(C^2)=0$ and $1\leq (L\cdot C) \leq p+2$.

It remains to show that the exceptions in Theorem~\ref{thm:verygeneric} correspond to  the exceptions 
in the statement in Theorem~\ref{thm:np}. For a curve $C\subseteq X$ which is smooth at the point $x$ and satisfies $1\leq (L\cdot C)\leq p+2$ one has  $(C^2)\leq 0$ by the Hodge index theorem  since $(L^2)\geq 5(p+2)^2$. Since we are on an abelian surface,  then automatically we have that $(C^2)=0$ and by adjunction this indeed forces $C$ to be an elliptic curve. 
\end{proof}

\begin{proof}[Proof of Theorem~\ref{thm:nopolygonabelian}]
To start with, \cite{KL}*{Proposition 3.1} gives the inclusion
\[
\Delta_{(E,z)}(\pi^*(B)) \cap \{2\}\times\RR  \dsubseteq  \{2\}\times [0,2], \textup{ for any } z\in E \ .
\]
Hence, according to  \cite{KL}*{Remark~1.9}, condition $(\ref{eq:nopolygonabelian})$ implies   $(\ref{eq:nopolygon1})$ in  Theorem~\ref{thm:nopolygon} for any $z\in E$ as $(B^2)>4$ by assumption. 

By Theorem~\ref{thm:nopolygon}  we know how to find a divisor $D\equiv B$ so that $\sJ(X,D)=\sI_{X,0}$ locally around a point. It remains to show that this equality 
in fact holds over the whole of $X$. 

Recall that $D$ is the image of a divisor $D'$ on $X'$, where (revisiting the proof of Theorem~\ref{thm:nopolygon}) 
\[
D' \ \equiv \ P+ \sum_{i=1}^{i=r}a_iE_i'+ 2E \ , 
\]
and  $\pi^*(B)-2E = P+\sum a_iE_i'$ is the appropriate Zariski  decomposition. 

Writing  $E_i$ for the image of $E_i'$ under  $\pi$,  the first step in the proof is to show the following claim:\\

\noindent
\textit{Claim:} Assume that  $o\neq y\in X$ such that  $\sJ(X,D)_y\neq \sO_{X,y}$. Then 
\begin{equation}\label{eq:claim}
\sum_{i=1}^{i=r}a_i\mult_o(E_i) \ < \ \sum_{i=1}^{i=r} a_i\mult_y(E_i)\ .
\end{equation}
\textit{Proof of Claim.} 
First observe that by \cite{PAGII}*{Proposition~9.5.13}, the condition $\sJ(X,D)_y\neq \sO_{X,y}$ implies  $\mult_y(D)\geq 1$. Since the morphism $\pi$ is an isomorphism 
around the point $y$,  considering $y$ as a point on  $X'$, we actually have  $\mult_y(D')\geq 1$. In the proof of Theorem~\ref{thm:nopolygon} we were able to write 
$P=A_k+\frac{1}{k}N$, where $A_k$ is ample and $N'$ is an effective $\QQ$-divisor  for any $k\gg 0$. 

Thus, we chose $D'=P'+\sum a_iE_i'+2E$, where $P'=A+\frac{1}{k}N'$, $A\equiv A_k$ is a generic choice, and $k\gg 0$. Since $A_k$ is ample, 
a generic choice of $A$ does not pass through the point $y$, in particular $\mult_y(P')\rightarrow 0$ as  $k\rightarrow \infty$. Since  $\mult_y(D')\geq 1$, 
this implies that
\[
\mult_y(\sum_ia_iE_i) \equ  \sum_ia_i\mult_y(E_i) \dgeq  1 \ . 
\]
As a consequence, it suffices to check that 
\[
\sum_i a_i\mult_o(E_i) \ < \ 1 \ .
\]
To this  end recall  that $E_i'=\pi^*E_i-\mult_o(E_i)E$, and therefore
\[
2 \equ  ((\pi^*B-2E)\cdot E) \equ  (P\cdot E) + \sum_ia_i(E_i'\cdot E) \equ  (P\cdot E)+\sum_ia_i\mult_o(E_i) \ .
\]
Because $(P\cdot E)$ is equal to the length of the vertical segment $\Delta_{(E,z)}(\pi^*B)\cap \{2\}\times\RR$ for any $z\in E$, 
$(\ref{eq:nopolygonabelian})$  implies that $\sum_ia_i\mult_o(E_i)<1$ as we wanted.\\
\textit{End of Proof of the Claim} \\

\noindent
We return to the proof of the main statement. We  will argue by contradiction and suppose that  there exists a point $y\in Y$ for which $\sJ(X;D)_x\neq \sO_{X,y}$. 
In other words, by the claim above, we have the inequality
\[
 \sum_{i}a_i\mult_o (E_i) < \sum_{i}a_i\mult_y (E_i) \ .
\]
This yields the existence of a curve $E_i$ with  $\mult_o(E_i) < \mult_y(E_i)$. 

By our assumptions $E_i'$ is a negative curve on $X'$; let $E_i^y$ denote  the proper transform  of $E_i$ with respect to the blow-up $\pi_y\colon X_y\to X$. 
Since $\mult_o(E_i)<\mult_y(E_i)$, we obtain 
\[
(E_i^y)^2 \equ  E_i^2-(\mult_y(E_y))^2 \ < \ E_i^2-(\mult_o(E_i))^2 \equ  E_i'^2\ <\ 0\ ,
\]
in particular, we deduce that $E_i^y$ remains a negative curve  on $X_y$, just  as $E_i'$ on $X'$. However, Lemma~\ref{lem:two points} shows  that $E_i$ must then be a 
smooth elliptic curve passing through $o$  and $y$. Therefore, $\mult_oE_i=\mult_yE_i=1$, which contradicts the  inequality above. 
\end{proof}

\begin{lemma}\label{lem:two points}
Let $X$ be an abelian surface,  $C\subseteq X$ a curve passing through two distinct points $x_1, x_2\in X$. If the proper transforms of $C$ for the respective  blow-ups 
of $X$ at the $x_i$ both become  negative curves,  then $C$ must be the smooth elliptic curve that is invariant under the translation maps $T_{x_1-x_2}$ or  $T_{x_2-x_1}$. 
\end{lemma}

\begin{proof}
Denote by $C_1$ and $C_2$  the proper transforms of $C$ with respect to  the blow-up of $X$ at $x_1$ and $x_2$, respectively. Aiming at a contradiction suppose that 
$T_{x_2-x_1}(C)\neq C$. Since  both proper transforms $C_1$ and $C_2$ are negative curves, one has 
\[
 0 > (C_1^2) \equ (C^2) - (\mult_{x_1}(C))^2\ \ \text{ and }\ \ 0 > (C_2^2) \equ(C^2)- (\mult_{x_2}(C))^2\ ,
\]
or,  equivalently $(C^2) <\min\{ \mult_{x_1}(C)^2 , \mult_{x_2}(C)^2\}$. On the other hand observe that 
\[
 (C^2) \equ  (C\cdot T_{x_2-x_1}(C)) \dgeq \mult_{x_2}(C)\cdot \mult_{x_2}(T_{x_2-x_1}(C))  \equ  \mult_{x_2}(C)\cdot \mult_{x_1}(C)\ ,
\]
for  $C$ is algebraically equivalent to its  translate $T_{x_2-x_1}(C)$. This is a contradiction, so  we conclude that $C$ is invariant under both of the translation maps $T_{x_1-x_2}$ or $T_{x_2-x_1}$, and is indeed an elliptic smooth curve. 
\end{proof}

\subsection{On the Koszul property of section rings}

 Still making use of the observations of Lazarsfeld--Pareschi--Popa, we produce a strong numerical criterion for the section ring $R(X;L)$ of an ample line bundle on an abelian surface
to be Koszul. Summarizing the proof of \cite{LPP}*{Proposition 3.1} one obtains:

\begin{proposition}$($ \cite{LPP}*{Proposition 3.1} $)$ \label{prop:LPP Koszul}
With notation as above assume that $X$ is an abelian surface, $L$ an ample line bundle on $X$. If there exists a rational number $0<c<1$ and an effective $\QQ$-divisor 
$F_0$ such that 
\begin{enumerate}
 \item $F_0\equiv \frac{1-c}{3}L$,
 \item $\sJ(X;F_0)=\sI_o$\ , 
\end{enumerate}
then $R(X,L)$ is Koszul. 
\end{proposition}

\begin{corollary}\label{cor:Koszul}
 Let $X$ be an abelian surface, $L$ an ample line bundle on $X$ with $(L^2)\geq 45$. If $X$ does not contain an elliptic curve $C$ with $(C^2)=0$ and  $1\leq (L\cdot C)\leq 3$, then $R(X,L)$ is Koszul. 
\end{corollary}
\begin{proof}
Our reasoning  is essentially the same as in the proof  of Theorem~\ref{thm:np}: take $p=1$ in Corollary~\ref{cor:very general}, and apply 
Theorem~\ref{thm:nopolygonabelian} to obtain the required divisor $F_0$. We conclude the proof with  Proposition~\ref{prop:LPP Koszul}. 
\end{proof}

\begin{remark}\label{rmk:Koszul}
 It is a classical result of Kempf \cite{Kempf} that $4L$ is Koszul for any ample line bundle $L$ on an abelian variety. As pointed out on \cite{LPP}*{p.2.},  the main interest for our result lies in the case when $L$ is not a (large) multiple of an ample line bundle, where for instance Kempf's theorem is not available. Furthermore, if $L^2\geq 3$ then Corollary~\ref{cor:Koszul} implies that $R(X,4L)$ is Koszul, thus recovering the result of Kempf.
 
 Regarding the bound of \cite{LPP}, our methods seem to imply a slightly weaker version. In reality looking carefully at the proof of Theorem~\ref{thm:verygeneric}, the results of \cite{LPP} are covered by \textit{Case}~$(1)$ there. The self-intersection $L^2$ was chosen to be large in order to force that only elliptic curves of small degree encode the syzygies of $L$ and the Koszul property of $R(X,L)$. It is only due to the large lower bound on $L^2$, that we don't seem to recover the bounds in \cite{LPP} straight from Corollary~\ref{cor:Koszul} or Theorem~\ref{thm:np}.

More importantly, note that our result confirms the Koszulness and the $(N_p)$ property of many ample line bundles with a small Seshadri constant and large self-intersection; 
the criterion we give is  easily checked in concrete examples.
\end{remark}

\begin{example}\label{eg:Koszul}
Here we present a class of  ample line bundles on a self-product of an elliptic curve where our criterion  verifies the Koszulness of the section ring, but \cite{LPP} does not.
For this, we will rely on computations from  \cite{BS}. 
 
In order to be in the situation of \cite{BS}*{Theorem 1}, let $C$ be an elliptic curve without complex multiplication, and let $L=a_1F_1+a_2F_2+a_3\Delta$ be an ample line
bundle on $C\times C$, where  the $F_i$'s  are  fibres of the two natural projections and $\Delta\subseteq C\times C$ stands for the class the diagonal. 
The self-intersection is then computed by 
\[
 (L^2) \equ 2a_1a_2 + 2a_1a_3+ 2a_2a_3 \ .
\]
We will  take $a_1,a_2,a_3>0$, hence  \cite{BS}*{Example 2.1} applies. Set $a_2=3$ and $a_3=2$, our plan is to take $a_1\gg a_2,a_3$; in any case if $a_1\geq 4$ then 
\[
 \epsilon(L;o) \equ \min\st{a_1+a_2,a_1+a_3,a_2+a_3} \equ 5\ ,
\]
and there is no elliptic curve of $L$-degree less than $5$ on $C\times C$. Our choice of $a_2$  and $a_3$ forces the line bundle  $L$ to be primitive, 
i.e. it is not a  multiple of any other line bundle on $C\times C$. Thus, neither \cite{Kempf} or \cite{LPP} apply. Moreover, if $a_1\geq 4$, then $(L^2)\geq 52\geq 45$, therefore $R(C\times C,L)$ is Koszul according to Corollary~\ref{cor:Koszul}.
\end{example}

\section{Low degree elliptic curves and syzygies}
 
In this section we give  the proof of the implication $(2)\Rightarrow (1)$ of Theorem~\ref{thm:np}.

\begin{theorem}\label{thm:counterexample}
Let $X$ be an abelian surface, $L$ a very ample line bundle on $X$. Suppose that $(L^2)\geq 4p+5$ and that there exists an elliptic curve $C\subseteq X$
with $1\leq (L\cdot C) \leq p+2$.  Then $L$ does not satisfy property $(N_p)$. 
\end{theorem}

\begin{proof}[Proof of $(2)\Rightarrow (1)$ of Theorem~\ref{thm:np}]
Assume that $X$ contains an elliptic curve $C$ with $C^2=0$ and $1\leq (L.C)\leq p+2$. Then either $L$ is not very ample, in which case it cannot satisfy 
property $(N_p)$ for $p\geq 0$, or it is, but then Theorem~\ref{thm:counterexample} applies.
\end{proof}

\begin{remark}\label{rmk:long proof}
Property $(N_p)$ is predominantly studied via vector bundle techniques. This turns out to be a feasible approach in our case as well,  by making use of 
 syzygy bundles  (cf. \cite{L1}) and Lazarsfeld--Mukai bundles  (introduced simultaneously in \cite{L1} and \cite{M1}). An extensive cohomology computation 
exploiting  these vector  bundles on  $X$ and $C$ leads to a proof  that $L$ does not satisfy property $(N_p)$ under the given conditions, primarily 
since  the restricted line bundle $L|_C$  does not satisfy the same property on the elliptic curve $C$ either. 

Along the way, one also needs to rely on some ideas  from \cite{P} via the observation  that $L|_C$ is actually numerically equivalent to the $p+2$-th 
power of a line bundle on the curve $C$. For a recent overview on the techniques involving syzygy bundles and Lazarsfeld--Mukai   bundles and their 
manifold applications  the reader is kindly referred to Aprodu's surveys \cite{A1} and \cite{A2}.
\end{remark}

In order to give a much shorter proof of  Theorem~\ref{thm:counterexample}, we will follow a  path suggested by Rob Lazarsfeld via restricting syzygies and 
invoke ideas of Eisenbud--Green--Hulek--Popescu from \cite{EGHP}. The kernel of the proof is the same as in Remark~\ref{rmk:long proof}, namely that 
low-degree elliptic curves have bad syzygies. 

Assuming  by contradiction that there are good syzygies on $X$,  one can take a plane and restrict the resolution of the ideal sheaf 
of $X$ to the elliptic curve $C$. This way one obtains a complex that is exact  away from a set of dimension one; diagram chasing will lead to the 
desired contradiction. This strategy  was first  introduced  in \cite{GLP}, and further developed in  \cite{EGHP},

We start by recalling  the following lemma.

\begin{lemma}[\cite{GLP}*{Lemma 1.6}]\label{lem:glp}
Let
\[
E_{\bullet} \ : \ \ldots \rightarrow E_{r-1} \rightarrow \ldots \rightarrow E_1\rightarrow E_0 \stackrel{\epsilon}\longrightarrow F\rightarrow 0
\]
be a complex of coherent sheaves on a projective variety $V$ of dimension $r$, with $\epsilon$ surjective. Assume that 
\begin{enumerate}
\item $E_{\bullet}$ is exact away from a set of dimension $1$.
\item For a given integer $1\leq m\leq r$, one has 
\[
H^i(V, E_0) = \ldots = H^i(V, E_{r-m}) \ = \ 0 \ \textup{ for } i>0 \ .
\]
\end{enumerate}
Then $H^i(V, F)=0$ for $i\geq m$.
\end{lemma}

\begin{proof}[Proof of Theorem~\ref{thm:counterexample}]
The case of projective normality was already discussed in Corollary~\ref{cor:normality}, 
so we can suppose  that $p\geq 1$. Note that in Corollary~\ref{cor:normality} $(L^2)\geq 20$ was assumed, but in fact the argument for the direction $(2)\Rightarrow (1)$ works whenever 
$(L^2)\geq 5$. This way, we will get to a contradiction by assuming $L$ satisfies property $(N_p)$ on $X$. 

Observe that  $L$  is projectively normal in particular, hence  very ample, and so $L$ gives rise to an embedding $X\subseteq \PP^N=\PP(H^0(X,L))$. 
As  the embedding above  is defined by the complete linear  series $|L|$,  in particular $H^0(\PP^N,\sI_{X|\PP^N}(1))=0$, the image is non-degenerate. 

Second, we point out that the natural restriction map
\begin{equation}\label{eq:restriction}
H^0(X,L) \ \stackrel{\textup{restr.}}{\longrightarrow} \ H^0(C,L|_C)
\end{equation}
is surjective. For this,  $H^1(X, L-C)=0$ would suffice,  which follows from  Kawamata--Viehweg vanishing once we show that  $L-C$ is big and nef. 
Since $(L^2)>2(L\cdot C)$, by the condition in Theorem~\ref{thm:counterexample},  \cite{PAGI}*{Theorem 2.2.15} implies  that $L-C$ is big. Nefness then  follows since  
$X$ is an abelian surface, therefore  any effective divisor is nef. 

Next, being the restriction of a very ample line bundle,  $L|_C$ defines an embedding $C\subseteq \PP^{p+1}=\PP(H^0(C,L|_C))$. Because  the restriction map in 
$(\ref{eq:restriction})$ is surjective,  $\PP^{p+1}$ can be seen as a $p+1$-dimensional plane $\Lambda\subseteq \PP^N$. Here we can assume that $(L\cdot C)=p+2$ by the induction hypothesis. 

The embedding $X\subseteq \PP^N$ was  non-degenerate therefore  the scheme-theoretic intersection $X\cap \Lambda$ is a subset of dimension  one (recall that $p+1<N$). 

Furthermore, by analyzing the exact sequence 
\[
 0 \lra \HH{0}{X}{L-C} \lra \HH{0}{X}{L} \lra \HH{0}{C}{L-C|_C} \lra 0\ ,
\]
we see that ${\mathfrak b}(|L-C|) = (\sI_{X\cap\Lambda/X}\colon \sI_{C/X}) \subseteq \sO_X$, in particular, $X\cap \Lambda = C$ if and only if the complete linear series $|L-C|$ 
is base-point free. 

We have already dealt with the case $p=0$, hence we will for the moment assume $p\geq 1$. As shown in  \cite{GrP}*{Section 1.2}, $L-C$ is base-point free by Reider's theorem whenever it is 
big and nef with $(L-C)^2\geq 5$. This is implied by the conditions $p\geq1$, $(L^2)\geq 4p+5$, $(L\cdot C)\leq p+2$, and $(C^2)=0$.

The rest of the proof follows that of \cite{EGHP}*{Theorem 1.1}, with the difference  that there the plane is of dimension at most $p$ (and 
hence  one obtains strong statements about the regularity of the ideal sheaf $\sI_{X\cap\Lambda|\Lambda}$ of the embedding $X\cap \Lambda\subseteq \Lambda$), whereas 
in our case the plane is of larger dimension; in any case, by the same token we will obtain that $H^2(\PP^{p+1},\sI_{X\cap\Lambda|\Lambda})=0$, 
which will give us the required contradiction.

Recall that aiming at a contradiction we assumed that $L$ satisfies property $(N_p)$ for some $p\geq 1$. By definition, one has a resolution of the ideal sheaf
\[
E_{\bullet}\ : \ \ldots \rightarrow E_{p+1}\rightarrow E_{p}\rightarrow \ldots \rightarrow E_{2}\rightarrow E_{1}\rightarrow \sI_{X|\PP^N}\rightarrow 0 \ ,
\]
where $E_i=\oplus \sO_{\PP^N}(-i-1)$ for any $i=1,\ldots ,p$. We tensor this resolution by $\sO_{\Lambda}$ and obtain a complex
\[
E_{\bullet}|_{\Lambda} \ : \ \ldots \rightarrow E_{p+1}\otimes \sO_{\Lambda}\rightarrow \oplus \sO_{\Lambda}(-p-1)\rightarrow \ldots \rightarrow \oplus \sO_{\Lambda}(-3)\rightarrow \oplus \sO_{\Lambda}(-2)\rightarrow \sI_{X|\PP^N}\otimes \sO_{\Lambda}\rightarrow 0
\]
that is exact outside the intersection $\Lambda \cap X$, i.e. away from a set of dimension $1$. 

Applying directly Lemma~\ref{lem:glp} and the automatic vanishing we have of higher cohomology of line bundles on projective space yields
\[
H^m(\Lambda, \sI_{X|\PP^N}\otimes \sO_{\Lambda}) \ = \ 0, \textup{ for any } m\geq 2 \ .
\]
Consider now the short exact sequence
\[
0\rightarrow \sI_X\cap \sI_{\Lambda}/\sI_X\cdot\sI_{\Lambda} \ \rightarrow \ \sI_{X|\PP^N}\otimes \sO_{\Lambda} \ \rightarrow \ \sI_{X\cap\Lambda|\Lambda}\rightarrow 0 \ .
\]
The kernel $\sK\deq \sI_X\cap \sI_{\Lambda}/\sI_X\cdot\sI_{\Lambda}$ is supported on the one-dimensional scheme $X\cap\Lambda$, so $H^m(\Lambda, \sK)=0$ for any $m\geq 2$. 
We then deduce that 
\[
H^m(\Lambda, \sI_{X\cap\Lambda|\Lambda}) \ = \ 0, \textup{ for any } m\geq 2 \ .
\]
Under the assumption $p\geq 1$ the  short exact sequence
\[
0\rightarrow \sI_{X\cap\Lambda|\Lambda} \rightarrow \sO_{\Lambda}\rightarrow \sO_{X\cap \Lambda}\rightarrow 0
\]
gives  that $H^1(X\cap\Lambda, \sO_{X\cap\Lambda})=0$. In particular, one arrives at  $H^1(C, \sO_C)=0$, which contradicts the fact that $C$ is an elliptic curve.
\end{proof}

\section{On a conjecture of Gross and Popescu}

In their quest for understanding the moduli space of abelian surfaces with a polarization of type $(1,d)$, Gross and Popescu \cite{GrP} 
formulated  the following conjecture at the end of their article. 

\begin{conjecture}\label{conj:grosspopescu}
The homogeneous ideal of an embedded $(1,d)$-polarized abelian surface is generated by quadrics and cubics whenever  $d\geq 9$.
\end{conjecture}

The goal of this section is to prove the conjecture under the slightly weaker numerical assumption  $d\geq 23$. 

\begin{theorem}\label{thm:grosspopescu}
Let $X$ be an abelian surface, $L$ a very ample line bundle  of type $(1,d)$ on $X$ for some $d\geq 23$. Then $L$ defines an embedding $X\subseteq \PP^N=\PP(H^0(X, L))$ whose ideal sheaf $\sI_{X|\PP^N}$ 
is globally generated by quadrics and cubics.
\end{theorem}

\begin{remark}\label{rem:grosspopescu}
The condition  $d\geq 23$ yields  $(L^2)\geq 46$; as $L$ was chosen to be very ample, we have  $(L \cdot C)\geq 3$ for any elliptic curve $C\subseteq X$, for   the restriction $L|_C$ is not very ample 
when $1\leq (L\cdot C)\leq 2$. Applying  Theorem~\ref{thm:np},  $(L^2)\geq 46$ implies  that $L$ satisfies property $(N_1)$, i.e. the ideal sheaf $\sI_{X|\PP^N}$ is generated by quadrics 
with the exception when there exists a Seshadri exceptional  elliptic curve $C\subseteq X$ with $(L\cdot C)=3$. 

If there exists an elliptic curve $C\subseteq X$ with $(L\cdot C)=3$, then  the restricted line bundle $L|_C$ embedds the curve $C\subseteq\PP^2$ as the zero locus of a degree $3$ polynomial. 
Thus, Theorem~\ref{thm:np} not only explains the need for cubic generators of the ideal sheaf asked in Conjecture~\ref{conj:grosspopescu}, 
but it also exposes the exceptions when these cubic generators are indeed needed.

With this said, in order to prove Theorem~\ref{thm:grosspopescu}, all that remains is to tackle the exceptional case along  the same path as taken for Theorem~\ref{thm:np}. More specifically, we will translate the 
condition of being  generated by quadrics and cubics to the existence of an effective divisor closely related to $L$, whose multiplier ideal is globally the maximal ideal of the origin,
and use the theory of infinitesimal Newton--Okounkov polygons to find such a divisor.  
\end{remark}

At first we will transcribe  the condition that the ideal sheaf is generated by quadrics and cubics into the vanishing of cohomology of some vector bundle on $X$ following ideas of Pareschi. To this end,  suppose that $L$ 
is a very ample line bundle on $X$ and denote by $M_L$ the kernel of the evaluation map
\[
0\rightarrow M_L\rightarrow H^0(X,L)\otimes \sO_X\rightarrow L\rightarrow 0\ .
\]

\begin{lemma}[\cite{PP}*{Section 6}]\label{lem:pp}
With notation as above, assume that
\[
H^1(X, M_L^{\otimes 2}\otimes L^{\otimes h}) \equ  0, \textup{ for all } h\geq 2 \ .
\]
Then the ideal sheaf arising from the embedding $X\subseteq \PP(H^0(X,L))$ is globally generated by quadrics and cubics.
\end{lemma}

Following Inamdar's work \cite{I}, let $\Delta_X\subseteq X\times X$ be the diagonal, and   consider the reduced algebraic set 
\[
\Sigma \deq  \{(x_0,x_1,x_2)| x_0=x_1 \textup{ and } x_0=x_2\} \ = \ \Delta_{0,1}\cup \Delta_{0,2} 
\]
on $X\times X\times X$. By  Lemma~\ref{lem:pp} and \cite{I}*{Lemma 1.1 and Lemma 1.3} we obtain the statement below. 

\begin{lemma}\label{lem:inamdar}
Let $h\geq 2$ be a natural number, and suppose that for any $h\geq 2$ we have 
\[
H^1(X\times X, \sI_{\Delta_X}\otimes L\boxtimes L^{\otimes h}) \ = \ H^1(X\times X\times X, \sI_{\Sigma}\otimes L^{\otimes h} \boxtimes L\boxtimes L) \equ   0 \ . 
\]
Then the ideal sheaf associated to the embedding $X\subseteq \PP(H^0(X,L))$ is globally generated by quadrics and cubics.
\end{lemma}

\begin{remark}\label{rem:56}
The vanishing of the first cohomology group in  the statement of Lemma~\ref{lem:inamdar} is already satisfied by $L$ if it is taken as 
in the hypothesis of Theorem~\ref{thm:grosspopescu}, since  the vanishing of the aforementioned cohomology group is equivalent to asking 
that $L$ defines a projectively normal embedding. But this happens  for the very ample line bundle $L$ in  
Theorem~\ref{thm:grosspopescu} according to   Theorem~\ref{thm:np}.
\end{remark}

Before we present the proof of Theorem~\ref{thm:grosspopescu},  we show how to translate our problem to finding an effective divisor  
related to $L$ with prescribed  singularities and numerics. Here we follow the line of thought of   \cite{LPP}.
 
\begin{proposition}\label{prop:lpp}
Let $L$ be a very ample line bundle on an abelian surface $X$ defining a projectively normal embedding. 
Suppose that for some rational number $0<c\ll 1$ there exists an effective $\QQ$-divisor 
\[
D \ \equiv \ \frac{(1+c)}{3}L \ ,
\]
such  that $\sJ(X;D)=\sI_{X,0}$ over the whole of  $X$. Then the ideal sheaf of the embedding $X\subseteq \PP(H^0(X,L))$ is 
globally generated by quadrics and cubics.
\end{proposition}
\begin{proof}
Continuing the argument of  Remark~\ref{rem:56}, it suffices  to show the vanishing of 
$H^1(X^3, \sI_{\Sigma}\otimes L^{\otimes h} \boxtimes L\boxtimes L)$ for all $h\geq 2$.  We will do this with the help of 
Nadel vanishing; for this to work  we will need   to find an effective divisor $E$ on $X^3$ such that 
\[
\big(L^{\otimes 2}\boxtimes L\boxtimes L\big)\big(-E\big) \textup{   is ample} \ 
\]
and $\sJ(X^3, E)=\sI_{\Sigma}$.

We will construct the divisor $E$ from the divisor $D$ defined on $X$ for some $0<c\ll 1$, as explained in \cite{LPP}. 
So, let $\textup{pr}_i:X\times X\rightarrow X$ for $i=1,2$. Consider the map
\[
\delta : X\times X\times X\rightarrow X\times X, (x_0,x_1,x_2)\rightarrow (x_0-x_1,x_0-x_2) \ .
\]
Note that
\[
\Sigma \ = \ \delta^{-1}\Big(\textup{pr}_1^{-1}(0)\bigcup \textup{pr}_2^{-1}(0)\Big)
\]
scheme-theoretically; set $E\deq \delta^*\big(\textup{pr}_1^*(D)+\textup{pr}_2^*(D)\big)$. Still using  \cite{LPP},  
 we obtain  $\sJ(X^3,E)=\sI_{\Sigma}$.  In order for  Nadel vanishing to apply, it remains to show that
\[
L^{\otimes 2}\boxtimes L\boxtimes L \ - \ \frac{(1+c)}{3}\delta^*(L\boxtimes L) \textup{    is ample} \ 
\]
for any rational number $0<c\ll 1$, since 
\[
 L^{\otimes 2}\boxtimes L\boxtimes L \ - \ \frac{1+c}{3}\delta^*(L\boxtimes L) \equ \pi^*L + \frac13 N -\frac{c}{3}\delta^*(L\boxtimes L)\ .
\]
By  \cite{LPP} there exists a nef line bundle $N$ on $X^3$ having the property that 
\[
\delta^*\big(L\boxtimes L\big)\otimes N \equ L^{\otimes 3}\boxtimes L^{\otimes 3}\boxtimes L^{\otimes 3} \ ,
\]
hence it suffices to verify that $\pi_1^*(L)+N$ is indeed ample, where $\pi_1:X\times X\times X\rightarrow X$ stands for the first 
projection. 

Following  \cite{LPP}*{Proposition 2.1} we observe that
\[
\pi_1^*(L)+N \ = \ \pi_1^*(L) \ + \ b^*(L)\ + \ d_{12}^*(L) \ ,
\]
where $b, d_{12}: X\times X\times X\rightarrow X$ are the maps given by $b(x_0,x_1,x_2)=x_0+x_1+x_2$ and $d_{12}(x_0,x_1,x_2)=x_2-x_1$. 
Notice that the product map $\pi_1\times b\times d_{12}$ is a finite morphism, thus
\[
\pi_1^*(L)+N\equ \big(\pi_1\times b\times d_{12}\big)^*\big(L\boxtimes L\boxtimes L\big) \ ,
\]
hence $\pi_1^*(L)+N$ is ample as required. 
\end{proof}

\begin{proof}[Proof of Theorem~\ref{thm:grosspopescu}]
Based on the discussion in Remark~\ref{rem:grosspopescu} it suffices  to treat the 'exceptional' case, i.e. when 
there exists a Seshadri exceptional elliptic curve $C\subseteq X$ with $(L\cdot C)=3$. Assume that such a curve does indeed exist. 

Write   $B\deq \frac{1}{3}L$ (which is an ample $\QQ$-divisor class) and let $\pi: X'\rightarrow X$ be the blow-up of the origin with $E$ 
the exceptional divisor, and $\overline{C}$ the proper transform of $C$ with respect to  $\pi$. Note that  $\epsilon(B;0)=(B\cdot C)=1$
under the circumstances. 

First observe that the negative parts of $\pi^*(B)-tE$ cannot contain negative curves beside $\overline{C}$  for all $t\in [1,2)$. 
Otherwise, the same argument  as in the proof of Theorem~\ref{thm:verygeneric} \textit{Case 3.(b)} will force the area of the generic 
infinitesimal Newton--Okounkov polygon $\Delta_x(B)$ to be strictly less than $5/2$, contradicting the fact that  the area of this polygon 
in fact equals  $(B^2)\geq 5/2$.

Having established that the  curve $\overline{C}$ is the only one appearing in the negative part of $\pi^*(B)-tE$ for any $t\in [1,2)$, we 
know that 
\[
\pi^*(B)-2E \ = \ P_2 \ + \ \overline{C} 
\]
is  the Zariski decomposition of $\pi^*(B)-2E$ for a suitable nef divisor $P_2$. Next, \cite{LM}*{Theorem 6.4} implies  
that $(2,1)\in\Delta_{(E,z)}(\pi^*(B))$ for any point $z\in E$. Choosing any rational number $c>0$, this leads to 
\[
\textup{interior of }\big(\Delta_{(E,z)}(\pi^*(1+c)B) \ \bigcap \ \Lambda\big) \ \neq \ \varnothing, \forall z\in E\ .
\]
Thus, the divisor $(1+c)B$ satisfies condition~$(\ref{eq:nopolygon1})$, and  Theorem~\ref{thm:nopolygon} yields the existence of an 
effective $\QQ$-divisor $D\equiv (1+c)B$ with $\sJ(X;D)=\sI_{0}$ locally around the origin. Since  the length of the segment at $t=2$ 
is greater than  one, Theorem~\ref{thm:nopolygonabelian} shows that the equality holds over the whole of $X$. 
By Proposition~\ref{prop:lpp} we are done.
\end{proof}


\end{document}